\documentclass{amsart}
\usepackage{hyperref}
\usepackage{graphicx}
\usepackage{amsmath}
\usepackage{amssymb}
\usepackage{amsthm}
\usepackage{dsfont}
\usepackage{xcolor}
\usepackage{mathrsfs}
\usepackage{accents}
\usepackage{resmes}
\usepackage{comment}

\numberwithin{equation}{section}
\newtheorem{theorem}{Theorem}[section]
\newtheorem{lemma}[theorem]{Lemma}
\newtheorem{proposition}[theorem]{Proposition}
\newtheorem{corollary}[theorem]{Corollary}
\newtheorem{definition}[theorem]{Definition}

\newtheorem{remark}[theorem]{Remark}
\newtheorem{example}[theorem]{Example}

\newcommand{\dd}{\mathrm{d}}

\newcommand{\cF}{\mathcal{F}}

\newcommand{\cP}{\mathcal{P}}
\newcommand{\cQ}{\mathcal{Q}}
\newcommand{\cT}{\mathcal{T}}
\newcommand{\cW}{\mathcal{W}}
\newcommand{\bE}{\mathbb{E}}

\newcommand{\bR}{\mathbb{R}}

\DeclareMathOperator*{\argmin}{arg\,min}

\DeclareMathOperator{\proj}{proj}
\numberwithin{equation}{section}

\begin{document}

\title{Shape-constrained density estimation with Wasserstein projection}

\author{Takeru Matsuda}
\address{Department of Mathematical Informatics, University of Tokyo \& RIKEN Center for Brain Science}
\email{matsuda@mist.i.u-tokyo.ac.jp}

\author{Ting-Kam Leonard Wong}
\address{Department of Statistical Sciences, University of Toronto}
\email{tkl.wong@utoronto.ca}


\keywords{Shape-constrained density estimation, Wasserstein projection, log-concave density}
\date{}



\begin{abstract}
Statistical inference based on optimal transport offers a different perspective from that of maximum likelihood, and has increasingly gained attention in recent years. In this paper, we study univariate nonparametric shape-constrained density estimation via projection with respect to the $p$-Wasserstein distance, with a focus on the quadratic case $p = 2$. By considering shape constraints given by displacement convex subsets of the Wasserstein space, Wasserstein projection estimation is a convex optimization problem. We focus on two fundamental examples, namely non-increasing densities on $\bR_+ := [0, \infty)$ and log-concave densities on $\bR$. In each case, we prove structural properties of the Wasserstein projection estimator, propose a discretization which can be implemented by off-the-shelf solvers, and compare the projection estimator with the corresponding  maximum likelihood estimator. 
\end{abstract}

\maketitle

\section{Introduction} \label{sec:intro}

In this paper, we apply Wasserstein projection in optimal transport to study nonparametric shape-constrained density estimation. For theoretical and computational tractability, we focus on the univariate setting where displacement convexity in optimal transport is equivalent to ordinary convexity in the space of quantile functions.

Let $X_1, \ldots, X_n$ be independent samples from an unknown distribution $\mu^*$. Let $\cF$ be a given set of distributions, possibly infinite dimensional, that encodes the shape constraint; it represents our statistical model. The problem is to come up with an estimator $\hat{\mu}_n$ of $\mu^*$ with values in $\cF$. We allow the model $\cF$ to be misspecified, so it is possible that $\mu^* \notin \cF$. Two important  cases that will be studied in detail in this paper are (i) the space of non-increasing densities on $\bR_+$ (the monotone case, see Section \ref{sec:monotone}) and (ii) the space of log-concave densities on $\bR$ (the log-concave case, see Section \ref{sec:log.concave}).

A popular approach to this problem is maximum likelihood estimation. Assuming that each element of $\cF$ has a density with respect to the Lebesgue measure (to be identified with the distribution itself), the (nonparametric) {\it maximum likelihood estimator} (MLE) is defined as the solution to the optimization problem
\begin{equation} \label{eqn:nonparametric.MLE}
\max_{f \in \cF} \sum_{i = 1}^n \log f(X_i),
\end{equation}
provided that it is well-posed. In the monotone case, this leads to {\it Grenander's estimator} introduced in \cite{G56}. The log-concave case was studied in \cite{W02}. Both, and other, shape constraints have been intensively studied in the literature, see for example the monograph \cite{GJ14} and the references \cite{D24, HWR26, S18, SS18}. 

In this paper, we study an alternative approach to the estimation problem, by using {\it optimal transport} \cite{S15, V03, V08}. For $p \geq 1$, let $\cW_p(\mu, \nu)$ be the {\it $p$-Wasserstein distance} between distributions $\mu$ and $\nu$ on $\bR$ (necessary concepts in optimal transport will be reviewed in Section \ref{sec:OT}). Let $\mu_n := \sum_{i = 1}^n (1/n)\delta_{X_i}$ be the empirical distribution of the data. Our {\it Wasserstein projection estimator} is defined by
\begin{equation} \label{eqn:Wasserstein.projection.intro}
\hat{\mu}_n := \argmin_{\nu \in \cF} \cW_p(\nu, \mu_n), 
\end{equation}
which may be regarded as a special case of the {\it minimum Kantorovich estimator} \cite{BBR06}. In order that the Wasserstein projection estimator exists uniquely (see Theorem \ref{thm:Wasserstein.projection}), we assume (i) $p > 1$, (ii)  distributions in $\cF$ have finite $p$-th moment, and (iii) $\cF$ is closed with respect to $\cW_p$ and is {\it displacement convex} in the sense of McCann \cite{M94}. These conditions not only guarantee that \eqref{eqn:Wasserstein.projection.intro} is a convex optimization problem, but also accommodate a wide variety of shape constraints. In most of the paper, we restrict further to the case $p = 2$, where the Wasserstein projection satisfies a desirable Lipschitz property with respect to $\cW_2$. From this, finite sample performance of the estimator in Wasserstein distance can be established using known results on convergence of empirical distributions (Propositions \ref{prop:rate1} and \ref{prop:log.concave.bounds}).

Maximum likelihood estimation \eqref{eqn:nonparametric.MLE} and Wasserstein projection estimation \eqref{eqn:Wasserstein.projection.intro} differ in the choice of the underlying geometry on the space of probability measures. While maximum likelihood estimation may be regarded as a projection with respect to the Kullback--Leibler divergence, Wasserstein projection incorporates the underlying Euclidean geometry of the state space. This leads to different behaviors of the estimators, especially in the misspecified setting. 

Our main contributions in this paper are to establish general properties of the Wasserstein projection estimator in the context of shape-constrained density estimation, and (focusing on the case $p = 2$) prove structural properties in the monotone and log-concave settings:
\begin{itemize}
\item In the monotone case, the estimated density is piecewise constant  and compactly supported (Theorem \ref{thm:monotone.characterization}).
\item In the log-concave case, the estimated density is piecewise log-affine and compactly supported (Theorem \ref{thm:log.concave.characterization}).
\end{itemize}
Although these properties appear to be qualitatively similar to those of the maximum likelihood estimators, there are important differences. For example, the support of the estimated density is generally {\it not} equal to the convex hull of the data points, and the set of break points (of intervals where the density is constant or log-affine) may {\it not} be a subset of the data points. To give a quick example (see Example \ref{eg:log.concave.uniform} for details), suppose that the data $\mu_n = (1/2) \delta_{-1} + (1/2) \delta_1$ is uniformly distributed on $\{-1, 1\}$, and $\cF$ is the set of log-concave distributions. It is well known that the maximum likelihood estimator returns the uniform distribution $\mathrm{Unif}(-1, 1)$. Interestingly, the $2$-Wasserstein projection estimator returns $\mathrm{Unif}(-3/2, 3/2)$, which has a wider support. Even in our univariate setting, the proofs of the above structural results are significantly more delicate than those of the classical counterparts. Intuitively, this is because optimal transport incorporates the geometry of the state space. Extending these results to the multivariate setting is interesting but requires different techniques due to the curvature of the Wasserstein space.

To illustrate empirically the performance of the Wasserstein projection estimator, we formulate and implement tractable discretizations for monotone and log-concave density estimation, which in our setting are $L^2$ (or more generally $L^p$) projections onto certain closed convex cones in the space of quantile functions. Since quantile functions are monotone by definition, our projection problems are  mathematically related to {\it isotonic regression} (see for example \cite{BB72}), but are more involved because of the shape constraints. We implement these algorithms in R \cite{R}.\footnote{Codes are available at \url{https://github.com/tkl-wong/wasserstein-projection-estimation}.} We report several experiments, using simulated data, to compare and contrast the Wasserstein projection estimator with the maximum likelihood estimator.

Our work contributes to the growing literature of statistical inference based on optimal transport,\footnote{This is related to, but different from, inference of transport-based quantities (such as optimal transport maps and Wasserstein barycenters) from samples, and modeling of distributional data.} as opposed to classical methods especially maximum likelihood. Recently, an intriguing framework of {\it Wasserstein information geometry} has been developed in the parametric setting \cite{LZ19,A24,AM24}; of fundamental importance are the Wasserstein information matrix and the resulting Wasserstein--Cram\'{e}r--Rao inequality \cite{NM25,TJS25}. It is interesting to understand to what extent this theory extends to semiparametric and nonparametric settings. Density estimation under Wasserstein distance has been studied by many authors. In particular, minimax estimation of smooth densities in Wasserstein distance was studied in \cite{NB22, wang2022minimax}, but without shape constraints. Note that smoothness of the density is not assumed in this paper\footnote{A non-increasing density may have jump discontinuities, while a log-concave density is continuous (and differentiable except possibly at countably many points) in the interior of its support.}. Statistical properties of the Wasserstein projection estimator have been investigated in one-dimensional location-scale models \cite{AM22} and circular models \cite{OM25}. Shape constraints were considered in   \cite{cumings2017shape}, where the author adopted a computational approach and did not prove structural results of the kind obtained here. 

The rest of the paper is organized as follows. In Section \ref{sec:prelim}, we introduce the Wasserstein projection estimator after recalling some basic concepts in optimal transport, and establish general properties including consistency and affine equivariance. In Sections \ref{sec:monotone} and \ref{sec:log.concave}, we specialize respectively to monotone and log-concave density estimation, and prove our main results about structural properties of the Wasserstein projection estimator. Section \ref{sec:implementation} implements tractable discretizations of Wasserstein projection estimation in the monotone and log-concave settings. In Section \ref{sec:discussion} we conclude and discuss some directions for future research. 

\section{Optimal transport and Wasserstein projection}  \label{sec:prelim}

\subsection{Wasserstein geometry of univariate distributions} \label{sec:OT}
Let $\cP(\bR)$ be the set of Borel probability measures on the real line $\bR$. If $\mu \in \cP(\bR)$ and $T$ is a Borel measurable map from $\bR$ into itself, we let $T_{\#} \mu := \mu \circ T^{-1} \in \cP(\bR) $ be the {\it pushforward} of $\mu$ under $T$. Probabilistically, this means that if $X$ is a random variable distributed as $\mu$, then $T(X)$ is distributed as $T_{\#} \mu$. For $\mu \in \cP(\bR)$, we let $Q_{\mu}: (0, 1) \rightarrow \bR$ be its (left-continuous) {\it quantile function} defined by
\begin{equation} \label{eqn:quantile.function}
Q_{\mu}(u) := \inf\{ x \in \bR: \mu((-\infty, x]) \geq u\}, \quad u \in (0, 1).
\end{equation}
When it is convenient to refer to the boundary values (e.g.~when $\mu$ is compactly supported or equivalently $Q_{\mu}$ is bounded), we take $Q_{\mu}(0) := \lim_{u \downarrow 0} Q_{\mu}(u)$ and $Q_{\mu}(1) := \lim_{u \uparrow 1} Q_{\mu}(u)$ which are respectively the $\inf$ and $\sup$ of the support of $\mu$. By construction, we have
\begin{equation} \label{eqn:quantile.transform}
(Q_{\mu})_{\#} \mathrm{Unif}(0, 1) = \mu,
\end{equation}
where $\mathrm{Unif}(0, 1)$ is the uniform distribution (Lebesgue measure) on $(0, 1)$. Let
\begin{equation*}
\begin{split}
\cQ &:= \{ Q_{\mu} : \mu \in \cP(\bR)\} \\
&= \{ Q: (0, 1) \rightarrow \bR \text{ non-decreaseing and left-continuous}\}
\end{split}
\end{equation*}
be the set quantile functions. The mapping $\mu \mapsto Q_{\mu}$ defines a bijection from $\cP(\bR)$ onto $\cQ$. Its inverse is given by the pushforward  \eqref{eqn:quantile.transform}. The convenience of working with quantile functions is one of our main reasons for restricting to univariate distributions in this paper. 

For $p \in [1, \infty)$, let 
\[
\cP_p(\bR) := \left\{ \mu \in \cP(\bR): \int_{\bR} |x|^p \dd \mu(x) < \infty
\right\}
\]
be the space of probability measures on $\bR$ with finite $p$-th moment. The corresponding space of quantile functions is given by
\begin{equation*}
\begin{split}
\cQ_p &:= \{ Q_{\mu} \in \cQ : \mu \in \cP_p(\bR) \} = \cQ \cap L^p([0, 1]) \\
&= \left\{ Q \in \cQ: \|Q\|_p^p := \int_0^1 |Q(u)|^p \dd u < \infty  \right\}.
\end{split}
\end{equation*}


\begin{definition}[$p$-Wasserstein distance] \label{def:Wasserstein.distance}
Let $p \in [1, \infty)$. For $\mu, \nu \in \cP_p(\bR)$, we define the $p$-Wasserstein distance $\cW_p(\mu, \nu)$ between $\mu$ and $\nu$ by
\begin{equation} \label{eqn:p.Wasserstein.distance}
\cW_p(\mu, \nu) := \left( \inf_{\pi \in \Pi(\mu, \nu)} \int_{\bR \times \bR} |x - y|^p \dd \pi(x, y) \right)^{1/p},
\end{equation}
where $\Pi(\mu, \nu)$ is the set of probability measures $\pi$ on $\bR \times \bR$ whose first and second marginals are respectively $\mu$ and $\nu$. Elements of $\Pi(\mu, \nu)$ are called couplings of $(\mu, \nu)$.\footnote{By an abuse of terminology, we also call a pair $(X, Y)$ of random variables a coupling of $(\mu, \nu)$ if $(X, Y) \sim \pi \in \Pi(\mu, \nu)$.}
\end{definition}

For univariate distributions, the $p$-Wasserstein distance admits an explicit formula in terms of the quantile functions. 

\begin{proposition} [Isometry] \label{prop:Wasserstein.quantile}
Let $p \in [1, \infty)$. For $\mu, \nu \in \cP_p(\bR)$, we have 
\begin{equation}\label{eqn:Wasserstein.quantile}
\cW_p(\mu, \nu) = \| Q_{\mu} - Q_{\nu} \|_p = \left( \int_0^1 |Q_{\mu}(u) - Q_{\nu}(u)|^p \dd u \right)^{1/p}.
\end{equation}
In particular, $\cP_p(\bR)$ (equipped with the $p$-Wasserstein distance) and $\cQ_p$ (equipped with the $L^p$ distance) are isometric via the mapping $\mu \mapsto Q_{\mu}$.
\end{proposition}
\begin{proof}
See \cite[Theorem 2.10]{BL19}.
\end{proof}

A consequence of \eqref{eqn:Wasserstein.quantile} is the following. Given $\mu, \nu \in \cP_p(\bR)$, $p \in [1, \infty)$ and let $\pi \in \Pi(\mu, \nu)$ be the {\it comonotonic coupling} defined by the joint distribution of $(X, Y) = (Q_{\mu}(U), Q_{\nu}(U))$, where $U \sim \mathrm{Unif}(0, 1)$. In symbols, we have
\begin{equation} \label{eqn:monotone.coupling}
\pi = (Q_{\mu}, Q_{\nu})_{\#} \mathrm{Unif}(0, 1),
\end{equation}
where $(Q_{\mu}, Q_{\nu})$ denotes the mapping $u \mapsto (Q_{\mu}(u), Q_{\nu}(u))$ from $(0, 1)$ to $\bR^2$. Then $\pi$ is optimal for $\cW_p(\mu, \nu)$. In fact, by \cite[Theorem 2.9]{S15}, the comonotonic coupling is the unique optimal coupling when $p > 1$.

\begin{remark} [Convex costs]
More generally, one may consider optimal transport costs of the form
\[
\cT_c(\mu, \nu) := \inf_{\pi \in \Pi(\mu, \nu)} \int_{\bR \times \bR} c(x, y) \dd \pi(x, y),
\]
where $c: \bR \times \bR \rightarrow \bR_+$ is a given Borel measurable cost function. 
If $c(x, y) = h(x - y)$ where $h: \bR \rightarrow \bR_+$ is convex, we have 
\begin{equation} \label{eqn:convex.cost}
\cT_c(\mu, \nu) = \int_0^1 h \left( Q_{\mu}(u) - Q_{\nu}(u) \right) \dd u,
\end{equation}
and the comonotone coupling is optimal (uniquely so when $h$ is strictly convex). We recover \eqref{eqn:Wasserstein.quantile} by letting $c(x, y) = |x - y|^p$. Using \eqref{eqn:convex.cost}, some results of the paper can be formulated for general convex costs. For concreteness, in this paper we focus on the $p$-Wasserstein distance, especially the quadratic case $p = 2$.
\end{remark}

Given a subset $\cF$ of $\cP(\bR)$, we let
\begin{equation*} 
\cQ_{\cF} := \{ Q_{\mu} : \mu \in \cF \} \subset \cQ
\end{equation*}
be the set of quantile functions of elements of $\cF$.

\begin{definition}[Displacement convexity] \label{def:displacement.convex}
A subset $\cF$ of $\cP(\bR)$ is said to be displacement convex if $\cQ_{\cF}$ is a convex subset of $\cQ$ in the usual sense.
\end{definition}

\begin{remark}
Our definition of displacement convexity is consistent with the one introduced by McCann \cite{M97}. To see this, let $\mu_0, \mu_1 \in \cP(\bR)$ and let $\pi \in \Pi(\mu_0, \mu_1)$ be the comonotonic coupling of $(\mu_0, \mu_1)$ defined by \eqref{eqn:monotone.coupling}. For $t \in [0, 1]$, consider McCann's displacement interpolation given by
\begin{equation} \label{eqn:displacement.interpolation}
\mu_t = ((1 - t) Q_{\mu_0} + t Q_{\mu_1})_{\#} \mathrm{Unif}(0, 1) = (\Pi_t)_{\#} \pi,
\end{equation}
where $\Pi_t(x, y) := (1 - t) x + ty$ in the linear interpolation between the initial and final locations. By construction, the quantile function of $\mu_t$ is $Q_{\mu_t} = (1 - t) Q_{\mu_0} + t Q_{\mu_1}$. Thus, a set $\cF \subset \cP(\bR)$ is displacement convex in the sense of McCann if and only if $\cQ_{\cF}$ is convex.
\end{remark}

Our main examples in this paper are the space of non-increasing densities on $\bR_+ = [0, \infty)$ and the space of log-concave densities on $\bR$. These spaces will be analyzed in detail in Sections \ref{sec:monotone} and \ref{sec:log.concave}, respectively. Here we provide two other examples, the first of which has been studied in \cite{AM22}.

\begin{example} [Location-scale family] \label{eg:location.scale}
Let $\mu_0 \in \cP(\bR)$ be a fixed reference distribution such as the standard Gaussian. For $(m, \sigma) \in \bR \times \bR_+$, let $T_{m, \sigma}: \bR \rightarrow \bR$ be the affine transformation
\[
T_{m, \sigma}(x) = m + \sigma x.
\]
Consider the scale location family defined by $\cF = \{ (T_{m, \sigma})_{\#} \mu_0 : (m, \sigma) \in \bR \times \bR_+\}$. It is easy to see that 
\[
\cQ_{\cF} = \{ m + \sigma Q_{\mu_0}: (m, \sigma) \in \bR \times \bR_+ \},
\]
which is convex. Hence $\cF$ is displacement convex. 
\end{example}

\begin{example}[bi-Lipschitz quantile functions]
Given $0 < m < M < \infty$, let $\cF$ be the set of distributions whose quantile functions are bi-Lipchitz with constants $m$ and $M$:
\[
\cQ_{\cF} = \left\{ Q \in \cQ: \frac{1}{M} \leq \frac{Q(v) - Q(u)}{v - u} \leq \frac{1}{m} \text{ for } 0 \leq u < v \leq 1\right\}.
\]
It can be shown that $\mu \in \cF$ if and only if $\mu$ has a density $f$ which is bounded below and above on its support: $m \leq f(x) \leq M$. It is easy to see that $\cF$ is displacement convex. Note that elements of $\cF$ are compactly supported but there is no common support ($\cF$ is closed under translation). This family is irregular from the perspective of maximum likelihood but not from that of optimal transport. 
\end{example}

Next, we state a basic projection theorem which underlies our estimation procedure. We assume $p \in (1, \infty)$ to ensure uniqueness of the projection, see Example \ref{eg:non.unique} below.

\begin{theorem} [Wasserstein projection] \label{thm:Wasserstein.projection}
Let $p \in (1, \infty)$. Let $\cF$ be a non-empty subset of $\cP_p(\bR)$ which is displacement convex and closed with respect to $\cW_p$. Then there exists a mapping $\proj_{\cF}^{(p)}: \cP_p(\bR) \rightarrow \cF$ such that for $\mu \in \cP_p(\bR)$, $\proj_{\cF}^{(p)} \mu$ is the unique minimizer of the projection problem
\begin{equation} \label{eqn:Wasserstein.projection}
 \inf_{\nu \in \cF} \cW_p(\nu, \mu).
\end{equation}

Moreover, when $p = 2$, $\proj_{\cF} := \proj_{\cF}^{(2)}$ is $1$-Lipschitz, that is, 
\begin{equation} \label{eqn:Lipschitz}
\cW_2 (\proj_{\cF} \mu, \proj_{\cF} \nu) \leq \cW_2 (\mu, \nu), \quad \mu, \nu \in \cP_2
(\bR).
\end{equation}
\end{theorem}
\begin{proof}
Identifying each $\nu \in \cF$ with its quantile function $Q_{\nu} \in \cQ_{\cF} \subset L^p([0, 1])$, the projection problem \eqref{eqn:Wasserstein.projection} is equivalent to the $L^p$ projection problem
\begin{equation} \label{eqn:Lp.projection}
\inf_{Q \in \cQ_{\cF}} \|Q - Q_{\mu}\|_p^p = \inf_{Q \in \cQ_{\cF}} \int_{[0, 1]} |Q(u) - Q_{\mu}(u)|^p \dd u.
\end{equation}
Existence of an optimal $\nu$ follows from standard results in functional analysis. Uniqueness follows from the strict convexity of $Q \mapsto \| Q - Q_{\mu}\|_p^p$. When $p = 2$, $L^2([0, 1])$ is a Hilbert space. Being the metric projection onto a closed convex set of a Hilbert space, $\proj_{\cF}$ is $1$-Lipschitz. 
\end{proof}

Since \eqref{eqn:Lp.projection} is a convex optimization problem, we immediately obtain the following first-order condition, which is both necessary and sufficient.

\begin{corollary}[First-order condition] \label{cor:first.order.condition}
In the context of Theorem \ref{thm:Wasserstein.projection}, 
$\hat{\mu} \in \cF$ is equal to $\proj_{\cF}^{(p)} \mu$ if and only if
\begin{equation} \label{eqn:first.order.condition}
\left. \frac{\dd}{\dd \epsilon} \right|_{\epsilon = 0^+} \int_{[0, 1]} |Q_{\hat{\mu}} + \epsilon (Q_{\nu} - Q_{\hat{\mu}}) - Q_{\mu}|^p \dd u \geq 0, \quad \text{for all } \nu \in \cF.
\end{equation}
In particular, when $p = 2$, \eqref{eqn:first.order.condition} is equivalent to
\begin{equation} \label{eqn:first.order.condition.p2}
\int_{[0, 1]} (Q_{\hat{\mu}} - Q_{\mu})(Q_{\nu} - Q_{\hat{\mu}})\dd u \geq 0, \quad \text{for all } \nu \in \cF.
\end{equation}

\end{corollary}

The Lipschitz property \eqref{eqn:Lipschitz} is useful for studying the finite sample properties of the Wasserstein projection estimator. Unfortunately, the Lipschitz property generally does not hold when $p \neq 2$ (see \cite{BDC17} and the references therein). For this reason, in the rest of the paper we focus on the case $p = 2$.

\begin{example}[Non-uniqueness when $p = 1$] \label{eg:non.unique}
Let $\cF = \{ \delta_m : m \in \bR\}$ be the set of point masses. Clearly, $\cF$ is displacement convex. Let $\mu = (1/2) \delta_0 + (1/2) \delta_1$. It is easy to verify that
\[
\argmin_{\nu \in \cF} \cW_1(\nu, \mu) = \{ \delta_x : x \in [0, 1] \},
\]
so the minimizer is not unique. 
\end{example}

We note the following affine equivariance property.

\begin{proposition}[Affine equivariance] \label{prop:equivariance}
In the context of Theorem \ref{thm:Wasserstein.projection}, suppose $\cF$ is invariant under a nondegenerate affine transformation $T: \bR \rightarrow \bR$ of the form
\begin{equation*} 
T(x) = a + b x,
\end{equation*}
where $(a, b) \in \bR \times (\bR \setminus \{0\})$. That is, 
\begin{equation} \label{eqn:invariance}
\cF = T_{\#} \cF := \{ T_{\#} \nu : \nu \in \cF\}.
\end{equation}
Then, for any $\mu \in \cP_p(\bR)$ we have
\begin{equation} \label{eqn:equivariance}
\proj_{\cF}^{(p)} (T_{\#} \mu) = T_{\#} (\proj_{\cF}^{(p)} \mu).
\end{equation}
\end{proposition}
\begin{proof}
For $\mu \in \cP(\bR)$ and $u \in [0, 1]$, we have
\begin{equation*}
\begin{split}
Q_{T_{\#}\mu}(u) = \left\{\begin{array}{ll}
        a + b Q_{\mu}(u), & \text{if } b > 0;\\
        a + b Q_{\mu}(1 - u), & \text{if } b < 0.\\
        \end{array}\right.
\end{split}
\end{equation*}
Note that when $b < 0$ the quantiles are flipped. Regardless of the sign of $b$, for $\mu \in \cP_p(\bR)$ and $\nu \in \cF$ we have
\begin{equation*}
\begin{split}
\cW_p^p(T_{\#}\nu, T_{\#}\mu) = \int_0^1 | Q_{T_{\#}\nu}(u) - Q_{T_{\#} \mu}(u)|^p \dd u = |b|^p \cW_p^p(\nu, \mu).
\end{split}
\end{equation*}
Hence $\cW_p(T_{\#}\nu, T_{\#} \mu) = |b| \cW_p(\nu, \mu)$. It follows that
\[
T_{\#} \nu = \proj_{\cF}^{(p)} T_{\#}\mu \Leftrightarrow \nu = \proj_{\cF}^{(p)} \mu.
\]
This gives \eqref{eqn:equivariance} and the proposition is proved.

\end{proof}

\begin{remark}[Stochastic dominance] \label{rmk:stochastic.dominance}
Let $\mu, \nu \in \cP(\bR)$. We say that $\nu$ has (first-order) stochastic dominance over $\mu$, written $\mu \preceq \nu$, if $Q_{\mu} \leq Q_{\nu}$. This is equivalent to the existence of a coupling $(X, Y)$ of $\mu$ and $\nu$ such that $X \leq Y$ almost surely. In particular, we can take the comonotonic coupling \eqref{eqn:monotone.coupling}.

It is natural to ask whether the Wasserstein projection $\proj_{\cF}^{(p)}$ is monotone with respect to stochastic dominance. That is, whether 
\[
\mu \preceq \nu \Rightarrow \proj_{\cF}^{(p)} \mu \preceq \proj_{\cF}^{(p)} \nu.
\]
It turns out that this property does not generally hold. In particular, it does not hold for monotone and log-concave density estimation. An example for the later is given in Example \ref{eg:log.concave.same.mean}.
\end{remark}

\subsection{Wasserstein projection estimator}
Let $X_1, \ldots, X_n$ be i.i.d.~samples from an unknown distribution $\mu^* \in \cP(\bR)$. We let $\mu_n := \sum_{i = 1}^n (1/n)\delta_{X_i}$ be the {\it empirical measure} regarded as a random element of $\cP(\bR)$. 

Let $\cF$ be a given subset of $\cP(\bR)$, regarded as a statistical model. The set $\cF$ can be parametric or nonparametric. For example, if $\cF$ is a location-scale family (Example \ref{eg:location.scale}), then $\dim \cF = 2$ and $\cF$ can be parametrized by the location and scale parameters. Here, we are primarily interested in the nonparametric case. We allow the model to be misspecified, so $\mu^*$ needs not be an element of $\cF$.

A popular approach to nonparametric density estimation is maximum likelihood estimation. In this paper, we study instead estimation by Wasserstein projection. That is, our estimator $\hat{\mu}_n$ is defined as the distribution in $\cF$ which is closest to the empirical distribution $\mu_n$ in terms of a Wasserstein distance. For this purpose, we fix $p \in (1, \infty)$ and assume that $\cF \subset \cP_p(\bR)$. Given an empirical distribution $\mu_n$, our estimation problem is
\begin{equation} \label{eqn:Wasserstein.estimation}
\inf_{\mu \in \cF} \cW_p(\mu, \mu_n).
\end{equation}
By Theorem \ref{thm:Wasserstein.projection}, the optimization problem \eqref{eqn:Wasserstein.estimation} has a unique minimizer whenever $\cF$ is displacement convex and closed in $(\cP_p(\bR), \cW_p)$. This motivates the following definition.

\begin{definition}[Wasserstein projection estimator] \label{def:Wasserstein.density.estimator}
Let $p \in (1, \infty)$. Suppose $\cF$ is a displacement convex and closed subset of $\cP_p(\bR)$ with respect to $\cW_p$. Given an empirical distribution $\mu_n = \sum_{i = 1}^n (1/n) \delta_{X_i}$, we call
\begin{equation} \label{eqn:Wasserstein.density.estimator}
\hat{\mu}_n := \proj_{\cF}^{(p)} \mu_n
\end{equation}
the ($p$-)Wasserstein projection estimator with respect to $\cF$.
\end{definition}

Recall that we write $\proj_{\cF}^{(2)} = \proj_{\cF}$.

\begin{theorem}[Consistency with respect to $\cW_2$] \label{thm:consistency}
Let $p = 2$ and suppose that $\mu^* \in \cP_2(\bR)$. Let $\mu_n = \sum_{i = 1}^n (1/n) \delta_{X_i}$ where $X_1, X_2, \ldots$ are i.i.d.~samples from $\mu^*$. Then
\begin{equation} \label{eqn:Wasserstein.estimator.bound.1}
\cW_2(\hat{\mu}_n, \proj_{\cF} \mu^*) \leq \cW_2(\mu_n, \mu^*).
\end{equation}
In particular, we have
\begin{equation} \label{eqn:Wasserstein.a.s.convergence}
\lim_{n \rightarrow \infty} \cW_2 (\hat{\mu}_n, \proj_{\cF}\mu^*) = 0 \quad \text{almost surely}.
\end{equation}
\end{theorem}
\begin{proof}
The inequality \eqref{eqn:Wasserstein.estimator.bound.1} is an immediate consequence of the Lipschitz property (which holds when $p = 2$) in Theorem  \ref{thm:Wasserstein.projection}.

Since $\mu^* \in \cP_2(\bR)$ and $\mu_n$ is the empirical measure of the first $n$ samples, it is well known that $\cW_2(\mu_n, \mu^*) \rightarrow 0$ almost surely as $n \rightarrow \infty$; for a proof see \cite[Section 5.1.2]{RF18}. Now \eqref{eqn:Wasserstein.a.s.convergence} follows immediately from the bound \eqref{eqn:Wasserstein.estimator.bound.1}.
\end{proof}

From \eqref{eqn:Wasserstein.estimator.bound.1},  finite sample performance of the Wasserstein estimator can be deduced in terms of the $2$-Wasserstein distance between the empirical distribution $\mu_n$ and the true distribution $\mu^*$. Many results are available in the literature, and the convergence rates depend the assumed properties of $\mu^*$. We quote here two results, taken from \cite{BL19, fournier2015rate}, showing that a wide range of behaviors is possible. Bounds specific to log-concave distributions will be given in Proposition \ref{prop:log.concave.bounds} below. Convergence properties of the {\it density} of $\hat{\mu}_n$ (when exists) are beyond the scope of this paper, and is left for future research.

\begin{proposition} \label{prop:rate1} { \ } 
\begin{enumerate}
\item[(i)] Suppose $\mu^* \in \cP_q(\bR)$ for some $q > 2$. Then there exists a universal constant $C > 0$ such that
\[
\bE [ \cW_2^2(\mu_n, \mu^*) ] \leq C \left( \int |x|^q \dd \mu^*(x) \right)^{2/q} \left( n^{-\frac{1}{2}} + n^{-(1 - \frac{2}{q})} \right).
\]
\item[(ii)] Let $F$ be the distribution function of $\mu^*$ and $f$ be the density of the absolutely continuous component of $\mu^*$. If
\[
J_2 := \int_{-\infty}^{\infty} \frac{F(x) (1-F(x))}{f
(x)} \dd x < \infty,
\]
then
\[
\bE[\cW_2^2(\mu_n,\mu^*)] \leq \frac{2}{n+1} J_2.
\]
\end{enumerate}
\end{proposition}
\begin{proof}
See \cite[Theorem 1]{fournier2015rate} and \cite[Theorem 5.1]{BL19}.
\end{proof}

\begin{remark}
If it is known a priori that $\mu^*$ has a density which is smooth, a possible improvement is to start with the {\it wavelet estimator} in \cite{NB22} and then project its value onto $\cF$. In this case, structural properties of the estimator (in the spirit of Theorem \ref{thm:monotone.characterization} and Theorem \ref{thm:log.concave.characterization}) are more difficult to establish. Nevertheless, they may still be numerically computed, possibly after suitable discretization.
\end{remark}

\section{Monotone density estimation} \label{sec:monotone}

In this section, we study monotone density estimation under Wasserstein projection. Our main result is Theorem \ref{thm:monotone.characterization}, which shows that the estimated density is piecewise constant. We begin by formulating an appropriate displacement convex set for the monotonicity constraint. 

\begin{lemma} \label{lem:monotone.characterization}
Let $\mu \in \cP(\bR_+)$. Then $\mu$ has a density which is (a.e.) non-increasing on $\bR_+$ if and only if its quantile function $Q_{\mu}: (0, 1) \rightarrow \bR_+$ satisfies the following properties:
\begin{itemize}
\item[(i)] $Q_{\mu}$ is convex.
\item[(ii)] $Q_{\mu}$ is strictly increasing.
\item[(iii)] $\lim_{u \downarrow 0} Q_{\mu}(u) = 0$.
\end{itemize}
\end{lemma}
\begin{proof}
The proof requires a bit of work since we do not impose any regularity condition, other than monotonicity, on the density.\\
($\Rightarrow$) Suppose $\mu$ has a density $f_{\mu}$ which is non-increasing on $\bR_+$. Then, $\mu$ is necessarily concentrated on an open interval $I = (0, b)$, where $0 < b \leq \infty$, such that $f_{\mu} > 0$ on $I$. By \cite[Proposition A.18]{BL19}, the quantile function $Q_{\mu}$ of $\mu$ is strictly increasing and absolutely continuous on $(0, 1)$. Moreover, the chordal slope of $Q_{\mu}$ is given by
\begin{equation*} 
\frac{Q_{\mu}(v) - Q_{\mu}(u)}{v - u} = \frac{1}{v - u} \int_u^v \frac{1}{f_{\mu}(Q_{\mu}(s))} \dd s, \quad 0 < u < v < 1.
\end{equation*}
Since $1/(f_{\mu} \circ Q_{\mu})$ is non-decreasing, we have that $Q_{\mu}$ is convex on $(0, 1)$. Also, since $Q_{\mu}$ is an increasing bijection from $(0, 1)$ onto $(0, b)$, we have $\lim_{u \downarrow 0} Q_{\mu}(u) = 0$.

($\Leftarrow$) Suppose that $Q_{\mu}$ satisfies (i)--(iii). Then $Q_{\mu}$ is an increasing homeomorphism from $(0, 1)$ onto $(0, b)$, where $b  = \lim_{u \uparrow 1} Q_{\mu}(u)$. Note that $\mu$ is concentrated on $(0, b)$. It follows that $F_{\mu} = Q_{\mu}^{-1}$ is continuous and strictly increasing on $(0, b)$. 

Since $Q_{\mu}$ is convex on $(0, 1)$, for any $u \in (0, 1)$ the left and right derivatives $Q_{\mu}'(u-)$ and $Q_{\mu}'(u+)$ exist and coincide at all except countably many points. Moreover, both one-sided derivatives are positive since  $Q_{\mu}$ is convex and strictly increasing. We will show that the one-sided derivatives of the distribution $F_{\mu} = Q_{\mu}^{-1}$ exist everywhere on $(0, b)$. Let $u \in (0, 1)$ be arbitrary, and let $x = Q_{\mu}(u) \in (0, b)$. Then, we have
\begin{equation} \label{eqn:derivative.definition}
\lim_{h \downarrow 0} \frac{Q_{\mu}(u + h) - Q_{\mu}(u)}{h} = Q_{\mu}'(u+) > 0.
\end{equation}
For $h > 0$ such that $u + h \in (0, 1)$, let $\tilde{h} = \tilde{h}(h) = Q_{\mu}(u + h) - Q_{\mu}(u)$. Since $Q_{\mu}: (0, 1) \rightarrow (0, b)$ is an increasing homeomorphism, we have $h \downarrow 0$ if and only if $\tilde{h} \downarrow 0$. Since $F_{\mu} = Q_{\mu}^{-1}$, We may rewrite \eqref{eqn:derivative.definition} to get
\[
\lim_{\tilde{h} \downarrow 0} \frac{\tilde{h}}{F_{\mu}(x + \tilde{h}) - F_{\mu}(x+)} = Q_{\mu}'(u+) \Rightarrow F_{\mu}'(x+) = \frac{1}{Q_{\mu}'( F_{\mu}(x) +)}, \quad x \in (0, b).
\]
Thus $F_{\mu}$ is right differentiable at $x$. Similarly, we have
\[
F_{\mu}'(x-) = \frac{1}{Q_{\mu}'( F_{\mu}(x) -)}, \quad x \in (0, b).
\]
It follows from the convexity of $Q_{\mu}$ that $F_{\mu}$ is concave on $(0, b)$. Hence, $F_{\mu}$ is absolutely continuous and its derivative (density) is equal to
\begin{equation} \label{eqn:monotone.density}
f_{\mu}(x) = \frac{1}{Q_{\mu}'(F_{\mu}(x))},
\end{equation}
for all but countably many points on $(0, b)$. Hence $\mu$ has a density which is (a.e.) non-increasing.
\end{proof}

For $p > 1$, let $\mathring{\cF}_{m, p} \subset \cP_p(\bR_+)$ be the set of $\mu \in \cP_p(\bR_+)$ which has a non-increasing density on $\bR_+$. From Lemma \ref{lem:monotone.characterization}, $\mu \in \mathring{\cF}_{m, p}$ if and only if $Q_{\mu} \in \cQ_p$ and satisfies properties (i)--(iii) in Lemma \ref{lem:monotone.characterization}. Clearly, each of properties (i)--(iii) imposes a convex constraint on $Q_{\mu}$. It follows that $\mathring{\cF}_{m, p}$ is displacement convex in the sense of Definition \ref{def:displacement.convex}. To apply Theorem \ref{thm:Wasserstein.projection}, we require that the model is both displacement convex and closed in $(\cP_p(\bR_+), \cW_p)$. Since $\mathring{\cF}_{m, p}$ is not $\cW_p$-closed (see Example \ref{eg:relative.complement} below), for technical purposes we will use instead the $\cW_p$-closure of $\mathring{\cF}_{m, p}$ (also see Proposition \ref{prop:monotone.properties}  below).

\begin{proposition} \label{prop:F.m.p}
For $p > 1$, let $\cF_{m, p}$ be the set of $\mu \in \cP_p(\bR_+)$ whose quantile function satisfies the following properties:
\begin{itemize}
\item[(i)] $Q_{\mu}$ is convex.
\item[(ii)] $Q_{\mu}$ is non-decreasing.\footnote{Since $Q_{\mu}$ is a quantile function, it is non-decreasing by construction. We state the monotonicity condition explicitly to contrast it with property (ii) in Lemma \ref{lem:monotone.characterization}.}
\item[(iii)] $\lim_{u \downarrow 0} Q_{\mu}(u) = 0$.\end{itemize}
Then, $\cF_{m,p}$ is the closure of $\mathring{\cF}_{m, p}$ in $(\cP_p(\bR_+), \cW_p)$ and is displacement convex. 
\end{proposition}
\begin{proof}

We first show that each $\mu \in \cF_{m, p}$ is the $\cW_p$-limit of a sequence in $\mathring{\cF}_{m, p}$. Let $\mu \in \cF_{m, p}$. For each $n \geq 1$, let $\mu_n$ be such that
\[
Q_{\mu_n}(u) = Q_{\mu}(u) + \frac{1}{n}u, \quad u \in (0, 1).
\]
Since $u \mapsto \frac{1}{n}u$ is convex, bounded and strictly increasing, $\mu_n$ is an element of $\mathring{\cF}_{m, p}$. Moreover, by Proposition \ref{prop:Wasserstein.quantile}, we have
\[
\cW_p(\mu_n, \mu)^p = \int_0^1 |Q_{\mu_n}(u) - Q_{\mu}(u)|^p \dd u = \frac{1}{(p + 1)n^p} \rightarrow 0, \quad n \rightarrow \infty.
\]

Next, we show that $\cF_{m, p}$ is closed in $(\cP_p(\bR_+), \cW_p)$. By Proposition \ref{prop:Wasserstein.quantile} again, it suffices to show that $\cQ_{\cF_{m, p}}$ is closed in $L^p([0, 1])$. Suppose $(Q_n)_{n \geq 1}$ is a sequence in $\cQ_{\cF_{m, p}}$ converging to some $Q$ in $L^p([0, 1])$. Since $Q_n$ is non-decreasing and convex with $\lim_{u \downarrow 0} Q_n(u) = 0$, we may extend it to a non-decreasing convex function on $\bR$ by letting $Q_n(u) = 0$ for $u \leq 0$. Since $Q_n \rightarrow Q$ in $L^p([0, 1])$, there exists a subsequence $Q_{n'}$ along which $Q_{n'} \rightarrow Q$ almost everywhere on $[0, 1]$ (and hence on $(-\infty, 1)$. By \cite[Theorem 10.8]{R70}, $Q_n$ converges pointwise on $\bR$ to a convex function $\tilde{Q}$ which is  equal to $Q$ on $(0, 1)$. Clearly, $\tilde{Q}$ is non-decreasing. Also, we have $\tilde{Q}(0) = \lim_{n' \rightarrow \infty} Q_{n'}(0) = 0$. Furthermore, since $\tilde{Q}$ is continuous at $0$ by \cite[Theorem 10.1]{R70}, we have $\tilde{Q}(0) = \lim_{u \rightarrow 0^+} \tilde{Q}(u)$. Since $Q = \tilde{Q}$ on $(0, \infty)$, we have $Q \in \cQ_{\cF_{m, p}}$.

Finally, we observe that $\cF_{m, p}$ is displacement convex since $\cQ_{\cF_{m, p}}$, being the closure of the convex set $\cQ_{\mathring{\cF}_{m, p}}$ in $L^p([0, 1])$, is convex.
\end{proof}

By Theorem \ref{thm:Wasserstein.projection}, the Wasserstein projection estimator $\hat{\mu}_n = \proj_{\cF_{m, p}}^{(p)} \mu_n$ is well-defined. 
It is not difficult to show that the relative complement $\cF_{m, p} \setminus \mathring{\cF}_{m, p}$ consists of probability distributions $\mu \in \cP_p(\bR_+)$ of the form
\begin{equation} \label{eqn:monotone.mu.decomposition}
\mu = \mu_0 + \mu_1,
\end{equation}
where $\mu_0 = \mu(\{0\}) \delta_0$ has mass $\mu(\{0\}) \in [0, 1]$ at $0$, and $\mu_1$ is an absolutely continuous subprobability measure (this means $\mu_1(\bR_+) = 1 - \mu(\{0\}) \leq 1$) whose density is non-increasing on $\bR_+$ and has finite $p$-th moment. In this case $Q_{\mu}(u) = 0$ for $u \in (0, \mu(\{0\})]$ and $Q_{\mu}$ is strictly increasing on $[\mu(\{0\}), 1)$. By Proposition \ref{prop:monotone.properties} below, when the data points take values in $(0, \infty)$, the Wasserstein projection estimator takes values in $\mathring{\cF}_{m, p}$, and so has a density. 

\begin{example}[An element of $\cF_{m, p} \setminus \mathring{\cF}_{m, p}$] \label{eg:relative.complement}
Consider the probability distribution
\[
\dd \mu(x) = \underbrace{\alpha \dd \delta_0}_\text{$\dd \mu_0(x)$} + \underbrace{(1 - \alpha) I_{(0, 1)} \dd x}_\text{$\dd \mu_1(x)$},
\]
where $\alpha \in (0, 1)$. Note that
\[
Q_{\mu}(u) = \frac{1}{1 - \alpha} \max\{u - \alpha, 0\}, \quad u \in (0, 1).
\]
So $\mu \in \cF_{m, p}$. For each $n \geq 1$, let $\mu_n$ be the distribution on $\bR_+$ with density
\[
f_n(x) = 
\left\{\begin{array}{ll}
\alpha n, & \text{if } 0 < x < \frac{1}{n};\\
1 - \alpha, & \text{if } \frac{1}{n} < x < 1 + \frac{1}{n};\\
0, & \text{if } x > 1 + \frac{1}{n}.
\end{array}\right.
\]
Clearly $\mu_n \in \mathring{\cF}_{m, p}$ for $n \geq \frac{1 - \alpha}{\alpha}$. It is easy to verify that for any $p$ we have $\cW_p(\mu_n, \mu) \rightarrow 0$ as $n \rightarrow \infty$.
\end{example}

\begin{proposition}\label{prop:monotone.properties} 
Let $p > 1$, $\mu \in \cP_{p}(\bR_+)$ and $\hat{\mu} = \proj_{\cF_{m, p}}^{(p)}\mu$. 
\begin{itemize}
\item[(i)] If $\mu$ is supported away from zero, that is, $\mu([0, \epsilon]) = 0$ for some $\epsilon > 0$, then $\hat{\mu}$ is absolutely continuous and so is an element of $\mathring{\cF}_{m, p}$.
\item[(ii)] If $\mu$ is compactly supported, then so is $\hat{\mu}$.
\end{itemize}
In particular, if $X_i$ are i.i.d.~samples from an absolutely continuous distribution $\mu^* \in \cP_p(\bR)$, the Wasserstein density estimator $\hat{\mu}_n = \proj_{\cF_{m, p}} \mu_n$, where $\mu_n = \sum_{i = 1}^n (1/n) \delta_{X_i}$, is almost surely compactly supported and absolutely continuous. 
\end{proposition}
\begin{proof}
(i) Suppose that $\mu([0, \epsilon]) = 0$ for some $\epsilon > 0$. Then the quantile function $Q_{\mu}$ of $\mu$ is bounded below by $\epsilon$. We claim that if $\nu \in \cF_{m, p} \setminus \mathring{\cF}_{m, p}$ then it is not optimal for $\proj_{\cF_{m, p}} \mu$. From the discussion below \eqref{eqn:monotone.mu.decomposition}, there exists $\delta > 0$ such that $Q_{\nu}(u) = 0$ for $u \leq \delta$. There are two cases to consider:

{\it Case 1. $\nu = \delta_0$}. Then $Q_{\nu} \equiv 0$. Let $\tilde{\nu} = \mathrm{Unif}(0, \epsilon) \in \mathring{\cF}_{m, p}$ which has quantile function $Q_{\tilde{\nu}}(u) = \epsilon u$. Since $0 = Q_{\nu} < Q_{\tilde{\nu}} \leq Q_{\mu}$ on $(0, 1)$, we have
\begin{equation*}
\begin{split}
\cW_p^p(\tilde{\nu}, \mu) = \int_0^1 |\epsilon u - Q_{\mu}(u)|^p \dd u < \int_0^1 |0 - Q_{\mu}(u)|^p \dd u= \cW_p^p(\nu, \mu).
\end{split}
\end{equation*}

{\it Case 2. $\nu \neq \delta_0$}. Then $u_0 := \sup\{u : Q_{\nu}(u) = 0 \} \in (0, 1)$. For $u_1 \in (u_0, 1)$, consider $\tilde{\nu}$ defined in terms of the quantile function:
\begin{equation*}
\begin{split}
Q_{\tilde{\nu}}(u) &:= \max\left\{ Q_{\nu}(u_1) \frac{u}{u_1}, Q_{\nu}(u)\right\} \\
&= \left\{\begin{array}{ll}
Q_{\nu}(u_1) \frac{u}{u_1}, & \text{if } 0 \leq u \leq u_1;\\
Q_{\nu}(u), & \text{if } u_1 < u \leq 1.\\
\end{array}\right.
\end{split}
\end{equation*}
The last equality holds by convexity of $Q_{\nu}$. It is easy to verify that $\tilde{\nu} \in \mathring{\cF}_{m, p}$. By continuity of $Q_{\nu}$, for $u_1$ sufficiently close to $u_0$ we have $Q_{\nu}(u_1) \leq \epsilon$. For such a choice of $u_1$, we have
\begin{equation*}
\begin{split}
\cW_p^p(\tilde{\nu}, \mu) &= \int_0^{u_1} \left|Q_{\nu}(u_1) \frac{u}{u_1} - Q_{\mu}(u)\right|^p \dd u + \int_{u_1}^1 |Q_{\nu}(u) - Q_{\mu}(u)|^p \dd u \\
&< \int_0^1 |Q_{\nu}(u) - Q_{\mu}(u)|^p \dd u\\
&= \cW_p^p(\nu, \mu).
\end{split}
\end{equation*}

(ii) Suppose that $\mu$ is supported on the compact interval $[0, M]$ for some $M > 0$. We will show that $\proj_{\cF_{m, p}} \mu$ is supported on the interval $[0, 2M]$. It suffices to prove the following claim. Let $\nu \in \cF_{m, p}$ be supported on an interval strictly larger than $[0, 2M]$. Then, there exists $\tilde{\nu} \in \cF_{m, p}$ such that $\tilde{\nu}$ is supported on $[0, 2M]$ and $\cW_p(\tilde{\nu}, \mu) < \cW_p(\nu, \mu)$.

Let such an element $\nu$ be given. Define a probability distribution $\tilde{\nu}$ on $\bR_+$ by
\[
\dd \tilde{\nu}(x) = I_{[0, 2M]}(x) \dd \nu(x) + \underbrace{(1 - \nu([0, 2M]))}_\text{$\in (0, 1)$} \dd \delta_0(x).
\]
That is, we cut off the (non-zero) tail of $\nu$ beyond $2M$ and replace that by a mass at $0$.\footnote{Note that here it is possible that $\mu$ charges the point $0$. If not, we may use the construction in (i) to remove the mass at $0$.} By construction, $\tilde{\nu}$ is supported on $[0, 2M]$ and is an element of $\cF_{m, p}$.

We proceed to show that $\cW_p(\tilde{\nu}, \mu) < \cW_p(\nu, \mu)$. Let $\pi \in \Pi(\nu, \mu)
$ be the comonotonic coupling between $\nu$ and $\mu$ which is optimal for $\cW_p(\nu, \mu)$. By the disintegration theorem, we may write
\[
\pi(\dd x, \dd y) = \nu(\dd x) \pi(\dd y \mid x) 
\]
for some kernel $\pi(\dd y \mid x)$ on $\bR_+$. Define a coupling $\tilde{\pi} \in \Pi(\tilde{\nu}, \mu)$ by
\begin{equation} \label{eqn:coupling.construction}
\begin{split}
\tilde{\pi}(\dd x, \dd y) &= I_{[0, 2M]}(x) \nu(\dd x) \pi (\dd y \mid x) + \tilde{\gamma}(\dd x, \dd y),
\end{split}
\end{equation}
where $\tilde{\gamma}(\dd x, \dd y)$ is the unique (trivial) coupling between the positive measures
\[
(1 - \nu([0, 2M])) \dd \delta_0(x),
\]
which is concentrated at $0$, and the remainder
\[
\tilde{\mu}(\dd y) = \mu(\dd y) - \int_{[0, 2M]} \nu(\dd x) \pi(\dd y \mid x),
\]
which has the same mass $1 - \nu([0, 2M])$ and is supported on $[0, M]$. The coupling $\tilde{\pi}$ is suboptimal for $\cW_p(\tilde{\nu}, \mu)$ but is sufficient for the purposes.

Consider
\begin{equation*}
\begin{split}
&\cW_p^p(\tilde{\nu}, \mu) \\
&\leq \int_{[0, 2M] \times [0, M]} |x - y|^p \tilde{\pi}(\dd x, \dd y) \\
&= \int_{[0, 2M]} \left( \int_{[0, M]} |x - y|^p \pi(\dd y \mid x) \right) \nu(\dd x) + \int_{\{0\} \times [0, M]} |0 - y|^p  \tilde{\gamma}(\dd x, \dd y) \\
&\leq \int_{[0, 2M]} \left( \int_{[0, M]} |x - y|^p \pi(\dd y \mid x) \right) \nu(\dd x) + M^p (1 - \nu([0, 2M])).
\end{split}
\end{equation*}

On the other hand, for $x > 2M$ and any $y \in [0, M]$ we have
\[
|x - y|^p > M^p.
\]
It follows that
\[
\int_{[2M, \infty)} \left( \int_{[0, M]} |x - y|^p \pi(\dd y \mid x) \right) \nu(\dd x) > M^p (1 - \nu([0, 2M])).
\]
Hence
\[
\cW_p^p(\tilde{\nu}, \mu) < \int_{\bR_+ \times [0, M]} |x - y|^p \underbrace{\nu(\dd x) \pi(\dd y \mid x)}_\text{$\pi(\dd x, \dd y)$} = \cW_p^p(\nu, \mu),
\]
and the claim is proved.
\end{proof}

As the following example shows, explicit computation of the Wasserstein projection can be non-trivial even for simple distributions.

\begin{example} \label{eg:point.mass}
Let $p = 2$ and let $\mu = \delta_1$ be the point mass at $1$. For $b > 0$, let
 $\nu_b = \mathrm{Unif}(0, b)$ be the uniform distribution on $[0, b]$. Then
\begin{equation*}
\begin{split}
\cW_2^2(\nu_b, \mu) =  \int_0^1 (bu - 1)^2 \dd u 
= \frac{1}{3} (b^2 - 3b + 3).
\end{split}
\end{equation*}
The unique optimizer over $b > 0$ is $b^* = 3/2 = 1.5$. We claim that $\mathrm{Unif}(0, 3/2)$ is optimal over $\cF_{m, 2}$, that is, $\proj_{\cF_{m, 2}} \mu = \mathrm{Unif}(0, 3/2)$. We illustrate this numerically in Example \ref{eg:monotone1}.

We use the first-order condition in Corollary \ref{cor:first.order.condition}. By Proposition \ref{prop:monotone.properties}, $\proj_{\cF_{m, 2}} \mu$ is compactly supported. It suffices to test with $\nu \in \cF_{2,m}$ which is compactly supported. Then $Q = Q_{\nu}: [0, 1] \rightarrow \bR_+$ is convex, non-decreasing and $Q(0) = \lim_{u \downarrow 0} Q(u) = 0$. In particular, $Q$ is bounded. Note that $Q_{\mu}(u) = (3/2)u$. We claim that
\begin{equation} \label{eqn:point.mass.claim}
\begin{split}
&\left. \frac{\dd}{\dd \epsilon} \right|_{\epsilon = 0^+} \int_0^1 \left( \frac{3}{2}u + \epsilon \left(Q(u) - \frac{3}{2}u \right) - 1 \right)^2 \dd u \\
&= \int_0^1 2 \left(\frac{3}{2}u - 1\right) \left(Q(u) - \frac{3}{2} u \right) \dd u =  \int_0^1 (3 u - 2) Q(u) \dd u \geq 0.
\end{split}
\end{equation}
(Note that $\int_0^1 (\frac{3}{2} u - 1) u \dd u = 0$.) 

Since $Q$ is convex, it is absolutely continuous. Write
\[
Q(u) = \int_0^u q(s) \dd s
\]
for some measurable function $q: [0, 1] \rightarrow \bR_+$ (we may take $q$ to be the right derivative of $Q$ which exists everywhere). By Fubini's theorem, we have
\begin{equation*}
\begin{split}    
\int_0^1 (3 u - 2) Q(u) \dd u &= \int_0^1 (3u - 2) \left( \int_0^u q(s) \dd s\right)\dd u \\
&= \int_0^1 q(s) \left( \int_s^1 (3u - 2) \dd u \right) \dd s \\
&= \int_0^1 q(s) K(s) \dd s,
\end{split} 
\end{equation*}
where $K(s) = (3s - 1)(1 - s)/2$. Note that $K(s)$ is non-positive on $[0, 1/3]$, positive on $(1/3, 1]$, and integrates to $0$ on $[0, 1]$. Since $q \geq 0$ is non-decreasing, we have
\begin{equation*}
\begin{split}
\int_0^1 (3u - 1) Q(u) \dd u &= \int_0^{1/3} q(s) K(s) \dd s + \int_{1/3}^1 q(s) K(s) \dd s \\
&\geq \int_0^{1/3} q\Big(\frac{1}{3}\Big) K(s) \dd s + \int_{1/3}^1 q\Big(\frac{1}{3}\Big) K(s) \dd s \\
&= q \Big(\frac{1}{3}\Big)\int_0^1 K(s) \dd s = 0. 
\end{split}
\end{equation*}
This proves the claim, and we conclude that $\proj_{\cF_{m, 2}} \mu = \mathrm{Unif}(0, 3/2)$.
\end{example}

We arrive at the main result of this section.

\begin{theorem} \label{thm:monotone.characterization}
Let $p = 2$ and $\mu_n = \sum_{i = 1}^n (1/n) \delta_{x_i}$ where $x_1, \ldots, x_n > 0$. Then, the density of $\hat{\mu}_n = \proj_{\cF_{m, 2}}  \mu_n$, which is non-increasing, is compactly supported and piecewise constant with finitely many pieces. 
\end{theorem}

\begin{proof}
We will argue in terms of the quantile functions. Let $Q_0 := Q_{\mu_n}$ be the quantile function of the data. Note that $Q_0 > 0$ and is piecewise constant. By Proposition \ref{prop:monotone.properties}, $\hat{\mu}_n$ has a non-increasing density which is compactly supported on $\bR_+$. Let $\hat{Q} := Q_{\hat{\mu}_n}$ be its quantile function. To prove the theorem, we need to show that $\hat{Q}$ is piecewise affine (see \eqref{eqn:monotone.density}).

We will make use of the first order condition (Corollary \ref{cor:first.order.condition}) which is given by 
\begin{equation} \label{eqn:monotone.proof.FOC}
\int_0^1 (Q(u) - \hat{Q}(u) ) (Q_0(u)- \hat{Q}(u)) \dd u \leq 0, \quad Q \in \cQ_{\cF_{m, 2}}.
\end{equation}
From the definition of $\cF_{m, 2}$ (see Proposition \ref{prop:F.m.p}), one can show that $Q \in  \cQ_{\cF_{m, 2}}$ if and only if there exists $h \in L^2([0, 1])$ such that (i) $h \geq 0$ (a.e.), (ii) $h$ is non-decreasing (a.e.) and (iii) $Q(u) = \int_0^u h(s) \dd s$. Let $\mathcal{H}$ be the set of such functions $h$, and note that $\mathcal{H}$ is a closed convex cone in $L^2([0, 1])$.

Write $\hat{Q}(u) = \int_0^u \hat{h}(s) \dd s$ where $\hat{h} \in \mathcal{H}$. We now express \eqref{eqn:monotone.proof.FOC} in the form
\[
\int_0^1 \left( \int_0^u (h(s) - \hat{h}(s)) \dd s \right) (Q_0(u) - \hat{Q}(u)) \dd u \leq 0, \quad  h \in \mathcal{H}.
\]
By Fubini's theorem, this is equivalent to
\[
\int_0^1 (h(s) - \hat{h}(s)) \Big( \underbrace{\int_s^1 (Q_0(u) - \hat{Q}(u)) \dd u}_\text{$=: \phi(s)$} \Big) \dd s \leq 0, \quad  h \in \mathcal{H}.
\]
It follows that
\begin{equation*}
\int_0^1 h(s) \phi(s) \dd s \leq \int_0^1 \hat{h}(s) \phi(s) \dd s, \quad h \in \mathcal{H}.
\end{equation*}
That is, 
\[
\sup_{h \in \mathcal{H}} \int_0^1 h(s) \phi(s) \dd s = \int_0^1 \hat{h}(s) \phi(s) \dd s.
\]
Since $\mathcal{H}$ is a cone, the supremum must be $0$. Thus, we have
\begin{equation} \label{eqn:cone.inequality.2}
\int_0^1 h(s) \phi(s) \dd s \leq \int_0^1 \hat{h}(s) \phi(s) \dd s = 0, \quad h \in \mathcal{H}.
\end{equation}

Consider the function
\begin{equation} \label{eqn:Phi}
\Phi(t) := \int_t^1 \phi(s) \dd s, \quad t \in [0, 1].
\end{equation}
We claim that $\Phi \leq 0$. Indeed, for each $t$, consider $h \in \mathcal{H}$ defined by $h(s) = 0$ for $s \in (0, t)$ and $h(s) = 1$ for $s \in [t, 1]$. From \eqref{eqn:cone.inequality.2}, we have $\Phi(t) = \int_0^1 h(s) \phi(s) \dd s \leq 0$.

Next, consider the identity $\int_0^1 \hat{h}(s) \phi(s) \dd s = 0$. Integrating by parts, we have
\begin{equation}
\begin{split}
0 &= \int_0^1 \hat{h}(s) (-\Phi'(s)) \dd s = [- \hat{h}(s) \Phi(s)]_{0}^1 + \int_0^1 \Phi(s) \dd \hat{h}(s).
\end{split}
\end{equation}
where the last integral is interpreted in the sense of Lebesgue--Stieljes.

Since $\Phi(1) = 0$, we have
\[
0 = \hat{h}(0) \Phi(0) + \int_0^1 \Phi(s) \dd \hat{h}(s).
\]
Since $\hat{h} \geq 0$ and $\Phi \leq 0$, both terms are non-positive. Thus, both terms are zero. In particular, we have
\begin{equation*} 
\int_0^1 \Phi(s) \dd \hat{h}(s) = 0.
\end{equation*}
It follows that $\dd \hat{h} = 0$ (as a Radon measure) on the set $\{ s : \Phi(s) < 0\}$. In other words, $\hat{h}$ can only increase on the zero set $Z := \{s : \Phi(s) = 0\}$ of $\Phi$.

We claim that $Z$ is finite. If so, $\hat{h}$ can only increase (jump) at finitely many points. Thus, $\hat{Q}(u) = \int_0^u\hat{h}(s) \dd s$ is piecewise affine, implying that the density of $\hat{\mu}$ is piecewise constant with finitely many pieces. We prove this claim in two steps.

\medskip
{\it Step 1:} $Z$ does not contain any non-trivial open interval.\\
We argue by contradiction. Suppose $I = (a, b)$ is such an interval. Then
\[
\Phi'(u) = -\phi(u) = 0, \quad u \in I.
\]
Recall that $\phi(u) = \int_u^1 (Q_0(s) - \hat{Q}(s)) \dd s$. Since $\phi$ is constant on $I$, we see by differentiation that $Q_0(u) = \hat{Q}(u)$ a.e.~on $I$. But $\hat{Q}(u) = \int_0^u \hat{h}(s) \dd s$. Since $Q_0$ is piecewise constant, we have $\hat{h}(u) = 0$ a.e.~on $I$. But since $\hat{h}$ is non-decreasing and $\geq 0$, we must have $\hat{h}(u) = 0$ on $(0, b)$. This gives $\hat{Q}(u) = 0$ on $(0, b)$, and on $I$ in particular. This is a contradiction since we assumed that $Q_0 > 0$.

\medskip

{\it Step 2:} $Z$ does not contain any accumulation point. This implies, by compactness of $[0, 1]$, that $Z$ is finite.

Note that $\Phi$ is continuously differentiable with
\[
\Phi'(t) = -\phi(t) = -\int_t^1 (Q_0(u) - \hat{Q}(u)) \dd u.
\]
Since $\hat{Q}$ is convex and bounded, it is continuous. Also, $Q_0(u)$ is piecewise constant. Thus, the left and right second derivatives $\Phi''(t-)$ and $\Phi''(t+)$ exist everywhere on $(0, 1)$. Moreover, they agree (so $\Phi''(t)$ exists) and are equal to $Q_0(t) - \hat{Q}(t)$, except at finitely many points where $Q_0$ jumps.\footnote{There can be none if $Q_0$ is constant. In this case, $\hat{\mu}_n$ is a uniform distribution by Example \ref{eg:point.mass} (and equivariance under scaling).}

Suppose on the contrary that $Z$ contains an accumulation point $t^*$. Then, there exists a sequence $(t_k)_{k \geq 1} \in Z \setminus \{t^*\}$ (of distinct points) such that $t_k \rightarrow t^*$. Since $\Phi$ is continuous, $\Phi(t^*) = 0$. Passing to a subsequence if necessary, we may assume $(t_k)$ is monotone. We consider the case $t_k \uparrow t^*$; the other case $t_k \downarrow t^*$ is similar. 

Since $Q_0$ is piecewise constant, there exists $c > 0$ and $\delta > 0$ such that $Q_0 = c$ on $[t^* - \delta, t^*)$. Thus, $\Phi''(t) = c-\hat{Q}(t)$ on $[t^* - \delta, t^*)$. By Rolle's theorem, for each $k$ there exists $t_k' \in [t_k, t_{k+1}]$ such that $\Phi'(t_k') = -\phi(t_k') = 0$. Since $t_k' \uparrow t^*$, by continuity of $\Phi'$ we have that $\Phi'(t^*) = 0$. Applying Rolle's theorem again for $k$ sufficiently large such that $t_k' \in [t^*-\delta,t^*)$, there exists $t_k'' \in [t_k', t_{k+1}']$ such that $\Phi''(t_k'') = c-\hat{Q}(t_k'')= 0$. In particular, $\hat{Q}(t_k'') = c$. Since $t_k'' \uparrow t^*$ and $\hat{Q}$ is monotone, we have that $\hat{Q} = c$ on $I := [t_k'', t^*)$, for some $k$ sufficiently large. This implies that $\Phi'' = 0$ on $I$. Since $\Phi(t^*) = \Phi'(t^*) = 0$, we have $\Phi = 0$ on $I$. This is a contradiction since $Z$ does not contain a non-trivial interval by Step 1.
\end{proof}

\section{Log-concave density estimation} \label{sec:log.concave}
In this section, we turn to log-concave density estimation. Recall that a probability measure $\mu \in \cP(\bR)$ is said to be {\it log-concave} if it is either a point mass, or has a density $f$ with respect to the Lebesgue measure such that $\log f : \bR \rightarrow [-\infty, \infty)$ is concave. We denote by $\cF_{lc}$ the set of log-concave distributions on $\bR$. A log-concave distribution has moments of all orders. That is, we have
\begin{equation} \label{eqn:log.concave.moments.finite}
\cF_{lc} \subset \bigcap_{p \geq 1} \cP_p(\bR).
\end{equation}
See \cite[Section B.1]{BL19} for a discussion of this and related results. We begin by characterizing log-concavity in terms of the quantile function.

\begin{lemma} \label{lem:log.concavity}
Let $\mu \in \cP(\bR)$ and let $Q = Q_{\mu}$ be its quantile function. Then $\mu$ is log-concave if and only if 
\begin{itemize}
\item[(i)] $Q$ is constant; or
\item[(ii)] $Q$ is absolutely continuous and there exists a version of $Q'$ such that $Q' > 0$ and $1/Q'$ is concave on $(0, 1)$.
\end{itemize}
In particular, $\cQ_{lc} := \cQ_{\cF_{lc}}$ is the set of $Q \in \cQ$ that satisfies either (i) or (ii).
\end{lemma}
\begin{proof}
$(\Rightarrow)$ Suppose $\mu$ is log-concave. 

If $\mu$ is a point mass, then $Q$ is constant and (i) holds. Otherwise, 
\[
\dd \mu(x) =  f(x) I_{(a, b)}(x) \dd x
\]
for some interval $(a, b)$ and density $f$ such that $f > 0$ on $(a, b)$ and $\phi := -\log f$ is convex. In particular, $f$ is continuous on $(a, b)$ (see \cite[Theorem 10.1]{R70}). By \cite[Proposition A.18]{BL19}, for any $0 < u_0 < u_1 < 1$, we have
\[
Q(u_1) - Q(u_0) = \int_{u_0}^{u_1} \frac{1}{f(Q(s))} \dd s.
\]
Thus $Q$ is continuously differentiable on $(0, 1)$ with
\begin{equation} \label{eqn:Q.derivative}
Q'(u) = \frac{1}{f(Q(u))} = e^{\phi(Q(u))} > 0.
\end{equation}
Since $\phi$ is convex, its one-sided derivatives $\phi'(x-)$ and $\phi'(x+)$ exist for all $x \in (a, b)$. It follows that $1/Q'$ is one-sided differentiable, and we have
\[
\left(\frac{1}{Q'}\right)'(u \pm) = - e^{-\phi(Q(u))} \phi'(Q(u)\pm) e^{\phi(Q(u))} = -\phi'(Q(u)\pm).
\]
Since $\phi'(Q(u)+) \leq \phi'(Q(v)-)$ for $u < v$ from the convexity of $\phi$, we have that $1/Q'$ is concave.

($\Leftarrow$) Clearly, if (i) holds then $\mu$ is a point mass. Assume (ii). Let $h(u) := 1/Q'(u)$ which is positive and concave (hence continuous). Let $F = Q^{-1}$ be the distribution function of $\mu$. By the inverse function theorem, $F$ is differentiable and
\begin{equation} \label{eqn:log.concave.density.derivative}
f(x) := F'(x) = \frac{1}{Q'(F(x))} = h(F(x)), \quad x \in I := Q((-\infty, \infty)).
\end{equation}
Thus, the density $f$ is positive on the range of $Q$ which is an interval. Since the one-sided derivatives of $h$ exist, we may differentiate $f$ one-sidedly and get
\[
(\log f)'(x\pm) = h'(F(x)\pm).
\]
Since $h'$ is non-increasing, $\log f$ is concave and the theorem is proved.
\end{proof}

Next, we show that $\cF_{lc}$ is $\cW_p$-closed and displacement convex for any $p > 1$. Hence, the Wasserstein projection $\proj_{\cF_{lc}}^{(p)}$ is well-defined. We stress that displacement convexity of $\cF_{lc}$ only holds in the univariate setting. It was proved in \cite{santambrogio2016convexity} that the space of log-concave distributions on $\bR^d$ is {\it not} displacement convex when $d \geq 2$.

\begin{proposition}[Displacement convexity of log-concave distributions on $\bR$] \label{thm:log.concave}
The set $\cF_{lc} \subset \cP(\bR)$ of log-concave distributions on $\bR$ is displacement convex\footnote{From \cite[Footnote 1]{santambrogio2016convexity}, this result was
proved by Young-Heon Kim but his proof is not publicly available. For completeness we provide a self-contained proof using Lemma \ref{lem:log.concavity}.} and is closed in $(\cP_p(\bR), \cW_p)$ for any $p \in [1, \infty)$.
\end{proposition}

\begin{proof}
We first observe that $\cF_{lc}$ is closed in $(\cP_p(\bR), \cW_p)$ for $p \in [1, \infty)$. Suppose $\mu_n \in \cF_{lc}$ and $\cW_p(\mu_n, \mu) \rightarrow 0$ for some $\mu \in \cP_p(\bR)$. Then, $\mu_n$ converges weakly to $\mu$ by \cite[Theorem 6.9]{V08}. Since $\cF_{lc}$ is closed under weak convergence (see \cite[Proposition 3.6]{SW14}), we have $\mu \in \cF_{lc}$.

To show that $Q_{\cF_{lc}}$ is convex, it suffices to show that if $Q_1, Q_2 \in \cQ_{\cF_{lc}}$ are such that $Q_1', Q_2' > 0$ and $1/Q_1', 1/Q_2'$ are concave, then 
\[
\frac{1}{ (1 - \lambda) Q_1' + \lambda Q_2'}
\]
is concave for $\lambda \in [0, 1]$. Note that
\[
\frac{1}{ (1 - \lambda) Q_1' + \lambda Q_2'} = \frac{1}{ (1 - \lambda) \frac{1}{1/Q_1'} + \lambda \frac{1}{1/Q_2'}} =: \Phi\left(\frac{1}{Q_1'},\frac{1}{Q_2'}\right),
\]
where $\Phi$ is the weighted harmonic mean. Since the weighted harmonic mean is concave and is non-decreasing in each variable, $\Phi\left(1/Q_1', 1/Q_2'\right)$ is concave by the conservation properties (under composition) in \cite[Section 3.2.4]{BV04}.  
\end{proof}

In the rest of this section we restrict to the case $p = 2$. We proceed to study some properties of the Wasserstein projection $\proj_{\cF_{lc}}$, beginning with some examples.

\begin{example}[Centering] \label{eg:log.concave.same.mean}
For any $\mu \in \cP_2(\bR)$, the projection $\hat{\mu} = \proj_{\cF_{lc}} \mu$ has the same mean as that of $\mu$. To see this, apply the first-order condition \eqref{eqn:first.order.condition.p2} twice with $\nu$ being the pushforwards of $\hat{\mu}$ under the translation maps $x \mapsto x \pm 1$. Note that $\nu$ is log-concave and $Q_{\nu} - Q_{\mu} = \pm 1$. We obtain
\begin{equation*}
\begin{split}
&\int_0^1 (Q_{\hat{\mu}}(u) - Q_{\mu}(u)) \dd u = 0\\ &\Rightarrow \int_{\bR} x \dd \hat{\mu}(x) = \int_0^1 Q_{\hat{\mu}}(u) \dd u = \int_0^1 Q_{\mu}(u) \dd u = \int_{\bR} x \dd \mu(x).
\end{split}
\end{equation*}
\end{example}

\begin{example}[Uniform distribution on a two-point set] \label{eg:log.concave.uniform}
Let $\mu = (1/2)\delta_{-1} + (1/2)\delta_1$. We claim that $\proj_{\cF_{lc}} \mu = \mathrm{Unif}(-3/2, 3/2)$.  

Since $\mu$ is symmetric about $0$ and $\cF_{lc}$ is invariant under inversion (that is, $T_{\#} \cF_{lc} = \cF_{lc}$ where $T(x) = -x$), by  Proposition \ref{prop:equivariance} we have that $\proj_{\cF_{lc}} \mu$ is symmetric about $0$. Conditioning on $\bR_+$, we see that the projection problem $\proj_{\cF_{lc}} \mu$ is equivalent to $\proj_{\{\delta_0\} \cup \tilde{\cF}} \delta_1$, where $\tilde{\cF}$ is the collection of log-concave and non-increasing densities on $\bR_+$. It is easily seen (see Proposition \ref{prop:log.concave.basic.properties}(i) below) that $\delta_0$ is not optimal. From Example \ref{eg:point.mass}, $\mathrm{Unif}(0, 3/2)$ minimizes $\nu \mapsto \cW_2(\nu, \delta_1)$ over all non-increasing densities on $\bR_+$, log concave or not. Since $\mathrm{Unif}(0, 3/2)$ is log concave, it must be optimal for $\proj_{\tilde{\cF}} \delta_1$. It follows that $\proj_{\cF_{lc}} \mu = \mathrm{Unif}(-3/2, 3/2)$.

Note that if $\nu = \delta_1$ then $\proj_{\cF_{lc}} \nu = \nu = \delta_1$. Note that $\mu \preceq \nu$ but $\proj_{\cF_{lc}} \mu \not\preceq \proj_{\cF_{lc}} \nu$. It follows that $\proj_{\cF_{lc}}$ is not monotone with respect to stochastic dominance (see Remark \ref{rmk:stochastic.dominance}). 
\end{example}

When the true distribution is log-concave, the convergence rate in Wasserstein distance is parametric up to possibly a logarithmic factor. 

\begin{proposition} \label{prop:log.concave.bounds}
Let the true distribution $\mu^*$ be an element of $\cF_{lc}$. Consider the Wasserstein projection estimator $\hat{\mu}_n = \proj_{\cF_{lc}} \mu_n$ where $p = 2$ and $\mu_n = \sum_{i = 1}^n (1/n) \delta_{X_i}$ is the empirical measure of $n$ independent samples from $\mu^*$.
\begin{itemize}
\item[(i)] Let $\sigma \in \bR_+$ be the standard deviation of $\mu^*$ (it is always finite by \eqref{eqn:log.concave.moments.finite}). Then for $n \geq 2$ we have
\begin{equation} \label{eqn:log.concave.bound.1}
\bE\left[ \cW_2^2(\hat{\mu}_n, \mu^*) \right] \leq C \sigma^2 \frac{\log n}{n},
\end{equation}
where $C > 0$ is a universal constant independent of $\mu^*$.
\item[(ii)] Suppose that $\mu^*$ is supported on a compact interval $[a, b]$. Then for $n \geq 2$ we have 
\begin{equation} \label{eqn:log.concave.bound.2}
\bE\left[ \cW_2^2(\hat{\mu}_n, \mu^*) \right] \leq \frac{C (b - a)^2}{n + 1},
\end{equation}
where again $C > 0$ is a universal constant independent of $\mu^*$.
\end{itemize}
\end{proposition}
\begin{proof}
By \cite[Corollary 6.12]{BL19}, we have
\[
\bE\left[ \cW_2^2(\mu_n, \mu^*) \right] \leq C \sigma^2 \frac{\log n}{n}, \quad n \geq 2.
\]
The bound \eqref{eqn:log.concave.bound.1} follows from \eqref{eqn:Wasserstein.estimator.bound.1}. Similarly, the bound \eqref{eqn:log.concave.bound.2} follows from \cite[Corollary 6.11]{BL19} which applies when $\mu^*$ is compactly supported.
\end{proof}

Next, we prove the analogues of Proposition \ref{prop:monotone.properties} and Theorem \ref{thm:monotone.characterization} in the log-concave setting. 

\begin{proposition} \label{prop:log.concave.basic.properties}
Let $p = 2$, $\mu \in \cP_2(\bR)$ and $\hat{\mu} = \proj_{\cF_{lc}} \mu$.
\begin{itemize}
\item[(i)] If $\mu$ is not a point mass, then $\hat{\mu}$ is absolutely continuous.
\item[(ii)] If $\mu$ is compactly supported, then so is $\hat{\mu}$.
\end{itemize}
\end{proposition}
\begin{proof}
(i) We only use the property that $\mathrm{Uniform}(a, b) \in \cF_{lc}$ for all $a \leq b$. 

Let $Q$ be the quantile function of $\mu$. Assume $\mu$ is not a point mass, so that $Q$ is not a.e.~constant. From Example \ref{eg:log.concave.same.mean}, among all point masses $\delta_x$, $x \in \bR$, the one that minimizes $\cW_2(\delta_x, \mu)$ is the point mass at the mean $\int_{\bR} x \dd \mu(x)$ of $\mu$. By a translation if necessary, we may assume without loss of generality that $\mu$ has mean zero. Thus, it suffices to show that $\hat{\mu} \neq \delta_0$. 

For $\epsilon \geq 0$, let $\nu_{\epsilon} = \mathrm{Uniform}(-\epsilon, \epsilon)$. Its quantile function is given by
\[
Q_{\epsilon}(u) = -\epsilon + 2\epsilon u, \quad u \in [0, 1].
\]
We have
\begin{equation*}
\begin{split}
\cW_2^2(\nu_{\epsilon}, \mu) = \int_0^1 (-\epsilon + 2\epsilon u - Q(u))^2 \dd u.
\end{split}
\end{equation*}
It follows that
\begin{equation} \label{eqn:point.mass.lemma.proof}
\begin{split}
\left. \frac{\dd }{\dd \epsilon}\cW_2^2(\nu_{\epsilon}, \mu) \right|_{\epsilon = 0^+} 
  &= -4 \int_0^1 u Q(u) \dd u,
\end{split}
\end{equation}
since $\int_0^1 Q(u) \dd u = 0$.

Let $Z(u) = \int_0^u Q(s) \dd s$ which is convex ($Q$ is non-decreasing) and vanishes on the boundary: $Z(0) = Z(1) = 0$. Thus $Z(u) \leq 0$ on $[0, 1]$. From convexity, on $(0, 1)$ we have either $Z \equiv 0$ or $Z < 0$. Since $Q$ is not a.e.~constant (in which case the constant must be $0$), we have $Z < 0$ on $(0, 1)$. Integrating by parts, we have
\begin{equation}
\begin{split}
\int_0^1 u Q(u) \dd u &= \left[ u Z(u) \right]_0^1 - \int_0^1 Z(u) \dd u = - \int_0^1 Z(u) \dd u > 0,
\end{split}
\end{equation}
From \eqref{eqn:point.mass.lemma.proof}, we have
\[
\left.\frac{\dd }{\dd \epsilon}\cW_2^2(\nu_{\epsilon}, \mu)\right|_{\epsilon = 0^+} < 0.
\]
Hence $\nu_0 = \delta_0$ is not optimal. We conclude that $\hat{\mu}$ is not a point mass.

(ii) We prove the following one-sided statement which is slightly stronger. Suppose that $\mu$ is supported on $(-\infty, b]$ for some $b < \infty$, so that $Q_{\mu} \leq b$. We show that the support of $\hat{\mu} = \proj_{\mathcal{F}_{lc}} \mu$ is also bounded on the right. That is, if $\hat{Q} := Q_{\hat{\mu}}$ then $\sup_{u \in (0, 1)} \hat{Q}(u) < \infty$. The left tail can be handled similarly.

Suppose on the contrary that $\sup_{u \in (0, 1)} \hat{Q}(u) = \infty$, so the right tail of the density of $\hat{\mu}$ is unbounded on the right. We will construct $\tilde{\mu} \in \cF_{lc}$, with a bounded right tail, such that $\cW_2(\tilde{\mu}, \mu) \leq \cW_2(\hat{\mu}, \mu)$. Since $\sup_{u \in (0, 1)} \hat{Q}(u) = \infty$, for $R > b$ sufficiently large we have $Z_R := \hat{\mu}((-\infty, R]) > 0$. Note that $Z_R \uparrow 1$ as $R \uparrow \infty$. When $Z_R > 0$, let $\tilde{\mu}_R$ be $\hat{\mu}$ conditioned on $(-\infty, Z_R]$, and note that $\tilde{\mu}_R$ is log-concave. Its quantile function is given by
\[
Q_{\tilde{\mu}_R}(u) = \hat{Q}(Z_R u), \quad u \in (0, 1).
\]
Note that $\tilde{Q}_R \leq \hat{Q}(Z_R) < \infty$. 

We estimate $\cW_2^2(\hat{\mu}, \mu) - \cW_2^2(\tilde{\mu}_R, \mu)$ from below. Write
\begin{equation*}
\begin{split}
&\cW_2^2(\hat{\mu}, \mu) - \cW_2^2(\tilde{\mu}_R, \mu)\\
&= \int_0^1 (\hat{Q}(u) - Q_{\mu}(u))^2 \dd u - \int_0^1 (\hat{Q}(Z_Ru) - Q_{\mu}(u))^2 \dd u\\
&= \int_0^1 (\hat{Q}(u) - \hat{Q}(Z_Ru)) (\hat{Q}(u) + \hat{Q}(Z_Ru) - 2 Q_{\mu}(u)) \dd u \\
&\geq \int_0^1 (\hat{Q}(u) - \hat{Q}(Z_Ru)) (\hat{Q}(u) + \hat{Q}(Z_Ru) - 2b) \dd u,
\end{split}
\end{equation*}
where in the last line we used the inequalities $\hat{Q}(u) - \hat{Q}(Z_Ru) \geq 0$ ($\hat{Q}$ is non-decreasing and $u \geq Z_R u$) and $Q_{\mu} \leq b$.

Expanding and rearranging the integrals, we see that the above is equal to
\begin{equation*}
\begin{split}
&\int_0^1 \hat{Q}^2 (u) \dd u  - \int_0^1 \hat{Q}^2(Z_Ru) \dd u - 2b \left( \int_0^1 \hat{Q} (u) \dd u  - \int_0^1 \hat{Q}(Z_Ru) \dd u \right) \\
&= \int_0^1 \hat{Q}^2(u) \dd u - \frac{1}{Z_R} \int_0^{Z_R} \hat{Q}^2(u) \dd u - 2b \left( \int_0^1 \hat{Q}(u) \dd u - \frac{1}{Z_R} \int_0^{Z_R} \hat{Q}(u) \dd u \right)\\
&= \int_{Z_R}^1 \left( \hat{Q}^2(u) - 2b \hat{Q}(u) \right) \dd u + \frac{Z_R - 1}{Z_R} \int_0^{Z_R} \left( \hat{Q}^2(u) - 2b \hat{Q}(u) \right) \dd u.
\end{split}
\end{equation*}
Since $\hat{Q}(u) \geq R$ for $u \geq Z_R$, when $R \geq \max\{0, 2b\}$ the first term is bounded below by $R (R - 2b) (1 - Z_R)$. On the other hand, for $R$ sufficiently large such that $Z_R \leq \frac{1}{2}$, the second term is bounded above by
\[
2(1 - Z_R )K, \quad K :=  \int_0^1  \left( \hat{Q}^2(u) - 2|b| |\hat{Q}(u)| \right) \dd  u < \infty.
\]
(Recall that $\hat{Q} \in L^2([0, 1])$. It follows that for $R$ sufficiently large, we have
\begin{equation*}
\begin{split}
\cW_2^2(\hat{\mu}, \mu) - \cW_2^2(\tilde{\mu}_R, \mu) &\geq R(R - b) (1 - Z_R) - 2(1 - Z_R) K \\
&= (1 - Z_R) (R(R - b) - 2K),
\end{split}
\end{equation*}
which is positive when additionally $R(R - b) - 2K > 0$. Thus we may let $\tilde{\mu} = \tilde{\mu}_R$ for such a value of $R$.
\end{proof}

\begin{theorem} \label{thm:log.concave.characterization}
Let $p = 2$ and $\mu_n = \sum_{i = 1}^n (1/n) \delta_{x_i}$ where $x_1, \ldots, x_n \in \bR$, be an empirical measure, and let $\hat{\mu}_n = \proj_{\cF_{lc}} \mu_n$. If $\mu_n$ is a point mass, then $\hat{\mu}_n = \mu_n$. Otherwise, $\hat{\mu}_n$ has a compactly supported density which is piecewise log-affine with finitely many pieces.
\end{theorem}

\begin{lemma} \label{lem:h.piecewise.affine}
Let $\mu \in \cF_{lc}$ be non-degenerate. Then, the density of $\mu$ is piecewise log-affine if and only if the function $h := 1/Q_{\mu}'$, which is positive and concave by Lemma \ref{lem:log.concavity}, is piecewise affine.
\end{lemma}
\begin{proof}
Let $f$ be the log-concave density of $\mu$. From \eqref{eqn:log.concave.density.derivative}, on the open interval $\{x: f(x) > 0\}$, $\log f$ is one-sided differentiable with
\[
(\log f)'(x \pm) = h'(F(x) \pm),
\]
where $F$ is the distribution function of $\mu$. Since $F$ is continuous and strictly increasing on the range of $Q_{\mu}$, the proof is completed by observing that the following statements are equivalent:
\begin{itemize}
\item[(i)] $f$ is piecewise log-affine.
\item[(ii)] $\log f$ is piecewise affine.
\item[(iii)] $(\log f)'$ is piecewise constant.
\item[(iv)] $h'$ is piecewise constant
\item[(v)] $h$ is piecewise affine.
\end{itemize}
\end{proof}

\begin{proof}[Proof of Theorem \ref{thm:log.concave.characterization}]
If $\mu_n$ is a point mass, it is clear that $\hat{\mu}_n = \mu_n$. In the rest of the proof we assume that $\mu_n$ is not a point mass.  Let $Q_0$ be the piecewise constant quantile function of $\mu_n$.

Since $\mu_n$ is not a point mass, by Proposition \ref{prop:log.concave.basic.properties}, $\hat{\mu}_n$ has a compactly supported density which is log-concave. Let $\hat{Q} := Q_{\hat{\mu}_n}$ be the quantile function of $\hat{\mu}_n$, and let $\hat{h} := 1/\hat{Q}'$ which is positive and concave. Since $\hat{Q}$ is bounded, we have that $1/h \in L^1(0, 1)$. From Lemma \ref{lem:h.piecewise.affine}, the theorem is proved if we show that $\hat{h}$ is piecewise affine with finitely many pieces.

Write
\[
\hat{Q}(u) = \hat{a} + \int_0^u \frac{1}{\hat{h}(s)} \dd s,
\]
where $\hat{a} := \hat{Q}(0)$. Apply the first order condition \eqref{eqn:first.order.condition.p2} with
\[
Q(u) = a + \int_0^u  \frac{1}{h(s)} \dd s,
\]
where $a \in \bR$ and $h > 0$ is concave and $1/h \in L^1([0, 1])$. Let $\mathcal{H}$ denote the set of such functions. We obtain
\[
\int_0^1 \left(a + \int_0^u \frac{1}{h(s)} \dd s - \hat{a} - \int_0^u \frac{1}{\hat{h}(s)} \dd s    \right) (Q_0(u) - \hat{Q}(u)) \dd u \leq 0.
\]
From  Example \ref{eg:log.concave.same.mean}, we have $\int_0^1 Q_0(u) \dd u = \int_0^1 \hat{Q}(u) \dd u$. So, we may drop $\hat{a}$ and $a$, and use Fubini's theorem to get
\begin{equation} \label{eqn:log.concave.first.order.condition}
\int_0^1 \frac{1}{h(s)} \phi(s) \dd s \leq \int_0^1 \frac{1}{\hat{h}(s)} \phi(s) \dd s, \quad h \in \mathcal{H},
\end{equation}
where $\phi(s) := \int_s^1 (Q_0(u) - \hat{Q}(u)) \dd u$. Since $\mathcal{H}$ is a convex cone (see Theorem \ref{thm:log.concave}), we have that the right hand side of \eqref{eqn:log.concave.first.order.condition} is zero. 

For $h \in \mathcal{H}$ and $\epsilon \in [0, 1)$, we have $\hat{h} + \epsilon (h - \hat{h}) \in \mathcal{H}$, so that
\begin{equation}  \label{eqn:log.concave.first.order.condition2.prelimit}
\int_0^1 \frac{1}{\hat{h}(s) + \epsilon(h(s) - \hat{h}(s))} \phi(s) \dd s \leq  \int_0^1 \frac{1}{\hat{h}(s)} \phi(s) \dd s,
\end{equation}
and equality holds when $\epsilon = 0$. Taking the derivative at $\epsilon = 0+$, we have 
\begin{equation} \label{eqn:log.concave.first.order.condition2}
\int_0^1 -\frac{h(s) - \hat{h}(s)}{\hat{h}^2(s)} \phi(s) \dd s \leq 0,
\end{equation}
whenever the derivative can be taken under the integral sign.

\medskip

We organize the reminder of the proof in two steps.

\medskip

{\it Step 1.} $\hat{h}$ cannot be strictly concave on a non-trivial open subinterval $I$ of $[0, 1]$. Hence, any non-trivial open interval of $I$ contains a non-trivial interval $J$ on which $\hat{h}$ is affine.

Suppose towards a contradiction that $\hat{h}$ is strictly concave on an open interval $I$. Consider
\[
\phi'(s) = - (Q_0(s) - \hat{Q}(s)), \quad s \in I.
\]
Since $Q_0$ is piecewise constant and $\hat{Q}$ is (continuous and) strictly increasing, there exists a non-trivial interval $J = [a, b] \subset I$ on which $\phi'(s) \neq 0$. Since $\phi'$ is continuous, this implies that $\phi$ is either strictly increasing or strictly decreasing on $J$. Replacing $J$ by a smaller subinterval if necessary, we may assume either $\phi(s) > 0$ on $J$ or $\phi(s) < 0$ on $J$. 

Since $\hat{h}$ is strictly concave on $[a, b]$, for $u \in (a, b)$ we have the strict inequalities
\begin{equation*}
\begin{split}
&\frac{b - s}{b - a} \hat{h}(a) + \frac{s - a}{b - a} \hat{h}(b) < \hat{h}(s), \quad \text{and}\\
&\hat{h}(s) < \min\{\hat{h}(a) + \hat{h}'(a+)(s - a), \hat{h}(b) + \hat{h}'(b-)(s - b)\}.
\end{split}
\end{equation*}
If $\phi(s) > 0$ on $J$, define
\[
\tilde{h}(s) :=
\left\{\begin{array}{ll}
\hat{h}(s), & \text{if } s \notin J; \\
\frac{1}{2}h(s) + \frac{1}{2}\left(\frac{b - s}{b - a} \hat{h}(a) + \frac{s - a}{b - a} \hat{h}(b)\right), & \text{if } s \in J.
\end{array}\right.
\]
Otherwise (that is, if $\phi(s) < 0$ on $J$), define
\[
\tilde{h}(s) :=
\left\{\begin{array}{ll}
\hat{h}(s), & \text{if } s \notin J; \\
\min\{\hat{h}(a) + \hat{h}'(a+)(s - a), \hat{h}(b) + \hat{h}'(b-)(s - b)\}, & \text{if } s \in J.
\end{array}\right.
\]
We have $\tilde{h}\in \mathcal{H}$ and the limit from \eqref{eqn:log.concave.first.order.condition2.prelimit} to \eqref{eqn:log.concave.first.order.condition2} is justified in both cases. Since
\[
-(\tilde{h}(s) - \hat{h}(s))\phi(s) > 0, \quad s \in (a, b),
\]
it follows that 
\[
\int_0^1 -\frac{h(s) - \hat{h}(s)}{\hat{h}^2(s)} \phi(s) \dd s = \int_J  \frac{-(h(s) - \hat{h}(s)) \phi(s)}{\hat{h}^2(s)} \dd s > 0,
\]
contradicting the first-order condition \eqref{eqn:log.concave.first.order.condition2}.

\medskip

Let $\mathcal{J} = (J_{\alpha})_{\alpha}$ be the (countable) family of maximal open subintervals of $[0, 1]$, such that $\hat{h}$ is affine on each $J_{\alpha}$. Note that the family $(J_{\alpha})_{\alpha}$ is pairwise disjoint. From Step 1, the union $\bigcup_{\alpha} J_{\alpha}$ is dense in $[0, 1]$.

\medskip

{\it Step 2.} For any $u_0 \in [0, 1]$, there exists an open interval $I$ in $[0, 1]$, containing $u_0$, such that $I$ intersects at most finitely many $J_{\alpha}$. Using the compactness of $[0, 1]$, a standard topological argument shows that $[0, 1] \setminus \bigcup_{\alpha} J_{\alpha}$ is finite. Thus $(J_{\alpha})$ is a finite family and we have that $h_{\alpha}$ is piecewise affine with finitely many pieces.

If $u_0 \in J_{\alpha}$ for some $\alpha$, we may simply pick $I = J_{\alpha}$. Suppose that $u_0 \notin \bigcup_{\alpha} J_{\alpha}$. To obtain a contradiction, assume any open interval containing $u_0$ intersects infinitely many $J_{\alpha}$. Hence, there exists $((a_k, b_k))_{k = 1}^{\infty} \subset \mathcal{J}$, with $b_k - a_k \rightarrow 0$, such that either
\[
b_k < a_{k+1} \ \ \forall k \quad \text{and} \quad a_k, b_k \uparrow u_0,  
\]
or
\[
a_{k+1} < b_k \ \ \forall k  \quad \text{and} \quad a_k, b_k \downarrow u_0.
\]
We consider the former case; the other one is similar. 

Consider the value $\phi'(u_0) = -(Q_0(u) - \hat{Q}(u))$. If $\phi'(u_0) \neq 0$, we may find by continuity $\delta > 0$ such that $\phi'(s)$ and $\phi'(u_0)$ have the same nonzero sign on $[u_0 - \delta, u_0]$. If $\phi'(u_0) = 0$, since $Q_0$ is piecewise constant, there exists $\delta > 0$ such that $Q_0$ is constant on $[u_0 - \delta, u_0)$. But $\hat{Q}$ is strictly increasing. Hence, $-(Q_0(u) - \hat{Q}(u))$ is strictly increasing and we have $\phi'(s) < 0$ on $[u_0 - \delta, u_0)$. In all cases, $\phi$ is strictly monotone on $[u_0 -\delta, u_0)$ for some $\delta > 0$. Replacing $\delta$ by a smaller positive value if necessary, we have that $\phi > 0$ on $[u_0 -\delta, u_0)$ or $\phi < 0$ on $[u_0 - \delta, u_0)$.

Let $k_0$ be sufficient large such that $a_{k_0} > u_0 - \delta$. Since $\hat{h}'(a_k)$ is strictly decreasing in $k$, for $s \in (a_{k_0}, u_0)$ we have the strict inequalities
\begin{equation*}
\begin{split}
&\frac{u_0 - s}{u_0 - a_{k_0}} \hat{h}(a_{k_0}) + \frac{s - a_{k_0}}{u_0 - a_{k_0}} \hat{h}(u_0) < \hat{h}(s), \quad \text{and}\\
&\hat{h}(s) < \min\{\hat{h}(a_{k_0}) + \hat{h}'(a_{k_0}+)(s - a_{k_0}), \hat{h}(u_0) + \hat{h}'(u_0-)(s - u_0)\}.
\end{split}
\end{equation*}
Now we may argue as in Step 1 to construct a concave $\tilde{h}$ which perturbs $\hat{h}$ on $(a_{k_0}, u_0)$ and violates the first-order condition \eqref{eqn:log.concave.first.order.condition2}. This completes the proof of the theorem.

\end{proof}

\section{Implementation and examples} \label{sec:implementation}
In this section, we implement monotone and log-concave density estimation with Wasserstein projection upon convexity-preserving discretizations, and provide examples.\footnote{We refer the reader to the Github repository \url{https://github.com/tkl-wong/wasserstein-projection-estimation} for the codes and details of our implementation.} Our purpose here is simply to illustrate the empirical behaviors of the Wasserstein projection estimator; further work is needed to develop more efficient and scalable algorithms and, more importantly, an understanding of when estimators based on optimal transport may be more suitable for real data.\footnote{In a parametric framework (and under suitable regularity conditions), it was shown recently in \cite{TJS25} that the Wasserstein projection is asymptotically optimal, in a suitable sense, under infinitesimal additive perturbation of the original data set.}  

We focus on the case $p = 2$, but note that the algorithms can be modified for general $p \geq 1$ (when $p = 1$ uniqueness of the solution can no longer be guaranteed). We will be working with quantile functions to utilize convexity of the Wasserstein projection. Consider a data-set represented as an empirical distribution $\mu_n = \sum_{i = 1}^n (1/n) \delta_{x_i}$, or more generally $\mu_n = \sum_{i = 1}^n w_i \delta_{x_i}$ where $w_i$ is the weight of $x_i$.

Let a partition
\[
\Pi: 0 = u_0 < u_1 < \cdots < u_K = 1
\]
of $[0, 1]$ be given, and let $\Delta u_i := u_i - u_{i-1}$. In the following, we assume that $Q_0 := Q_{\mu_n}$ is piecewise constant with respect to $\Pi$. That is, for each $i$ there exists $y_i$ such that
\[
Q_0(u) = y_i, \quad u_{i-1} \leq u < u_i. 
\]
This is the case if $\Pi$ is a refinement of the uniform grid $(i/n)_{i = 0}^n$ (or its weighted analogue). Otherwise, we may discretize and let $y_i$ be the $100u_i\%$-quantile of $\mu_n$. The reason why discretization is necessary is that the sets of break points in Theorems \ref{thm:monotone.characterization} and \ref{thm:log.concave.characterization} generally do not correspond to subsets of the data points and their quantile probabilities. Further study of these break points is a natural direction for future research.

\subsection{Monotone density estimation} \label{sec:implement.monotone}
We optimize over the convex subset $\tilde{\cQ}_{m}$ of quantile functions in $\cQ_{\cF_{m, 2}}$ that are piecewise affine with respect to $\Pi$. The Wasserstein projection can be recast as a quadratic programming problem.

Let $Q \in \tilde{\cQ}_{m}$ and define $q_i = Q(u_i)$, $i = 0, \ldots, K$. From Lemma \ref{lem:monotone.characterization}, we always have $q_0 = 0$. The vector ${\bf q} := (q_i)_{i = 1}^K$ satisfies the monotonicity constraint
\begin{equation} \label{eqn:monotone.constraint1}
0 = q_0 \leq q_1 \leq \cdots \leq q_K
\end{equation}
as well as the convexity constraint
\begin{equation} \label{eqn:monotone.constraint2}
\frac{q_i - q_{i-1}}{u_i - u_{i-1}} \leq \frac{q_{i+1} - q_i}{u_{i+1} - u_i}, \quad i = 1, \ldots, K - 1,
\end{equation}
which are both linear. Conversely, any vector ${\bf q} = (q_i)_{i = 1}^K$ satisfying \eqref{eqn:monotone.constraint1} and \eqref{eqn:monotone.constraint2} defines an element $Q_{{\bf q}}$ of $\tilde{\mathcal{Q}}_m$ by linear interpolation.
Let $\nu_{{\bf q}} := (Q_{{\bf q}})_{\#} \mathrm{Unif}(0, 1)$. If $q_0 < q_1 < \cdots < q_K$, the density of $\nu_{{\bf q}}$ exists and is given on $[0, z_K]$ by
\begin{equation} \label{eqn:monotone.estimated.density}
f_{{\bf q}}(x) = \frac{u_i - u_{i-1}}{q_i - q_{i-1}}, \quad q_{i-1} < x < q_i.
\end{equation}
Clearly, $f_{{\bf q}}(x) = 0$ for $x > q_K$. We have
\begin{equation} \label{eqn:W2.monotone}
\begin{split}
\cW_2^2(\nu_{{\bf q}}, \mu_n) &= \|Q - Q_0\|_{L^2([0, 1]}^2 \\
&= \sum_{i = 1}^K \int_{u_{i-1}}^{u_i} \left( \frac{u_i - u}{u_i - u_{i-1}} q_{i-1} +  \frac{u - u_{i-1}}{u_i - u_{i-1}} q_i - y_i \right)^2 \dd u\\
&= \sum_{i = 1}^n \frac{\Delta u_i}{3} \left( q_{i-1}^2 + q_i^2 + q_{i-1}q_i - 3y_i (q_{i-1} + q_i) + 3y_i^2 \right),
\end{split}
\end{equation}
which is convex and quadratic in ${\bf q}$. Thus, the Wasserstein projection problem
\begin{equation} \label{eqn:QP}
\min_{{\bf q}}  \cW_2^2(\nu_{{\bf q}}, \mu_n) \quad \text{subject to \eqref{eqn:monotone.constraint1} and \eqref{eqn:monotone.constraint2}}
\end{equation}
is a quadratic program. We implement \eqref{eqn:QP} using the R package \texttt{quadprog} \cite{quadprog}. 

We will compare the Wasserstein projection estimator with Grenander's estimator \cite{G56}. It is well known that the density of Grenander's estimator is given by the slopes of the least concave majorant of the distribution function of $\mu_n$. Equivalently, the quantile function of Grenander's estimator is the greatest convex minorant of $Q_0$.\footnote{More precisely, we enforce the constraint that the quantile function is $0$ at $u = 0$. This guarantees that the estimated density is non-increasing on $\bR_+$.} We implement the greatest convex minorant, and hence Grenander's estimator, using the R package \texttt{fdrtool} \cite{fdrtool}.

\begin{figure}[t!]
\includegraphics[scale = 0.48]{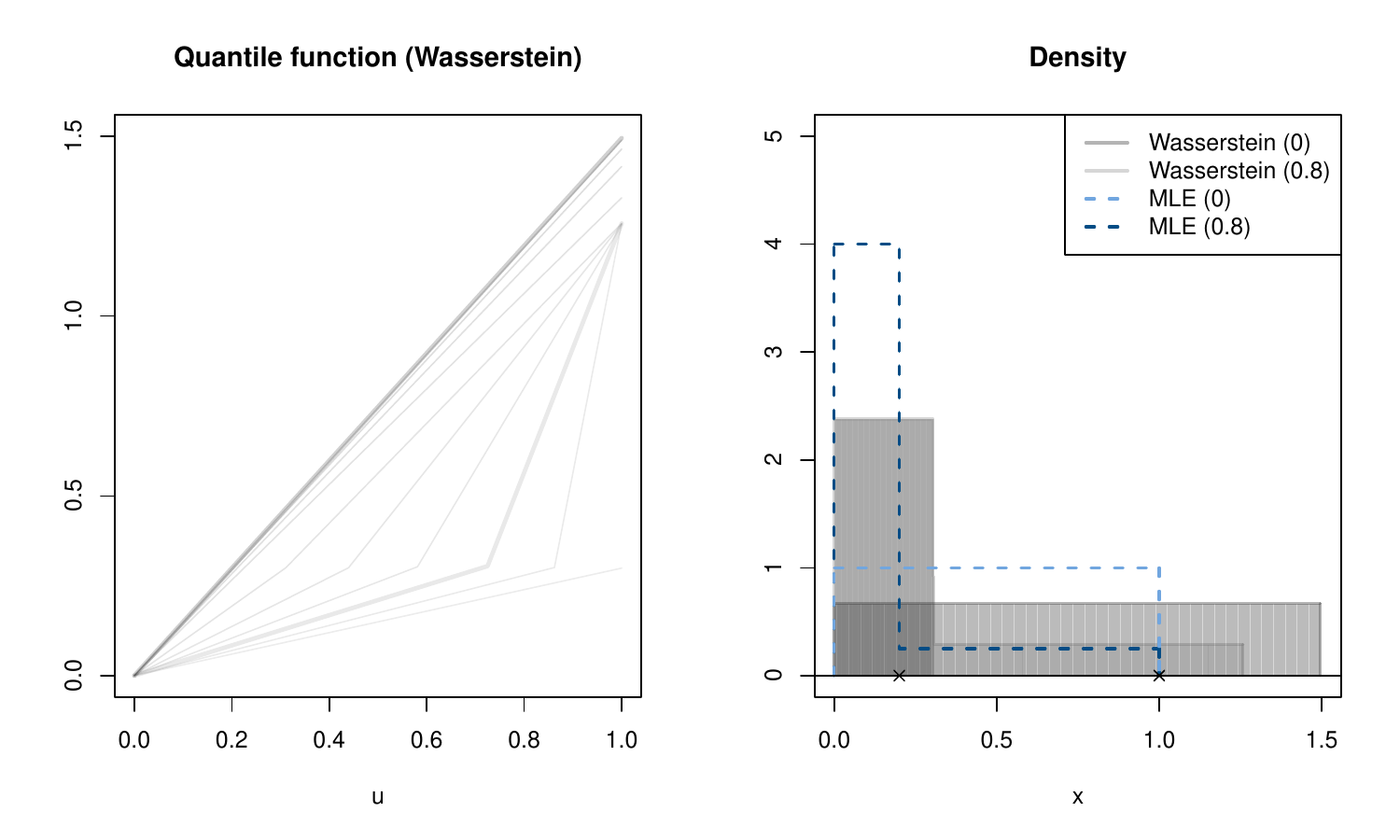}
\vspace{-0.4cm}
\caption{Left: Estimated quantile functions from the Wasserstein projection estimator, for the mixture distributions given in \eqref{eqn:monotone1} (Example \ref{eg:monotone1}). Right: Densities of the Wasserstein projection estimator (shaded) and Grenander's estimator (dashed), for two cases ($\lambda = 0, 0.8$) highlighted by thicker lines on the left panel. The support $\{0.2, 1\}$ of the data is shown by the crosses.} \label{fig:monotone1}
\end{figure}

\begin{example}[Mixture of two points] \label{eg:monotone1}
For $\lambda \in [0, 1]$, let $\mu_n$ be the mixture distribution
\begin{equation} \label{eqn:monotone1}
\mu_n = (1 - \lambda) \delta_{0.2} + \lambda \delta_1.
\end{equation}
For each $\lambda \in \{0, 0.1, \ldots, 0.9, 1\}$, we compute the Wasserstein projection estimator $\hat{\mu}_n$. We plot the quantile function of $\hat{\mu}_n$ in the left panel of Figure \ref{fig:monotone1}. 

When $\lambda = 0, 1$, $\mu_n$ reduces to $\delta_{0.2}$ and $\delta_{1}$ respectively. We see that $\hat{\mu}_n$ becomes respectively $\mathrm{Unif}(0, 0.3)$ and $\mathrm{Unif}(0, 1.5)$ as predicted by Example \ref{eg:point.mass} and equivariance (Proposition \ref{prop:equivariance}) under scaling. We see that a knot occurs for several intermediate values of $\lambda$. Nevertheless, its location is generally not equal to one of the quantile probabilities ($i/10$) of the data points. For the cases shown, the support of $\hat{\mu}_n$ is strictly larger than that of Grenander's estimate, namely $[0, \max_i x_i]$. In the right panel of Figure \ref{fig:monotone1} we plot the piecewise constant density of $\hat{\mu}_n$ for the cases $\lambda = 0.2, 1$. Likewise, the break points for the density do not form a subset of the data points. Again, this is different from the density of Grenander's estimator, which is also shown.
\end{example}

\begin{figure}[t!]
\includegraphics[scale = 0.48]{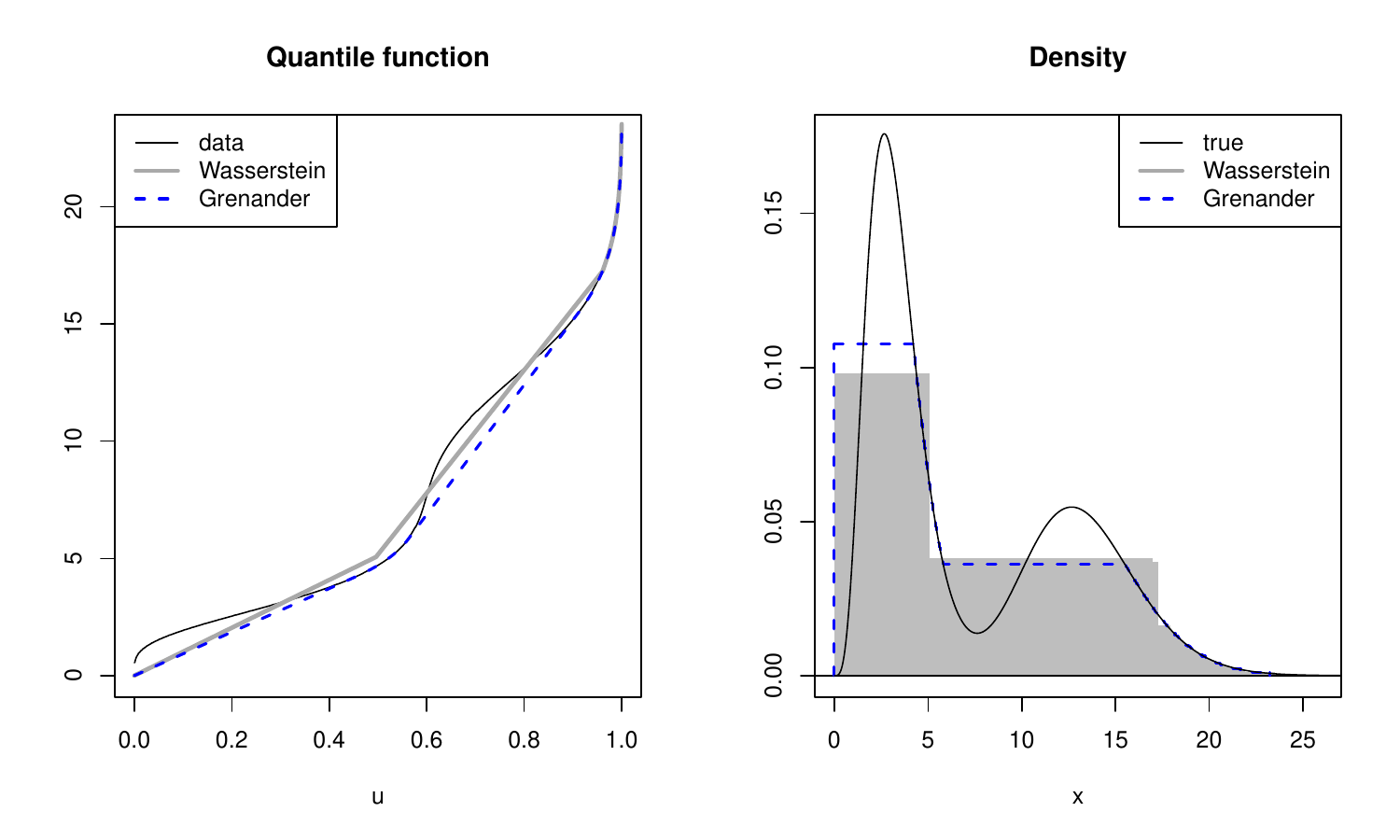}
\vspace{-0.4cm}
\caption{Left: Data and estimated quantile functions in the context of Example \ref{eg:monotone3}. Right: True and estimated densities.} \label{fig:monotone3}
\end{figure}

\begin{example}[A misspecified case] \label{eg:monotone3}
Consider $\mu_n = \sum_{i = 1}^{1000} (1/1000) \delta_{x_i}$, where $x_i$ is the $100v_i$\% quantile of $\mu^* = 0.6 \mathrm{Gamma}(5, 1.5) + 0.4 \mathrm{Gamma}(22, 1.5)$ (the second slot represents the rate parameter) and the sequence $(v_i)_{i = 1}^{1000}$ is a uniform grid between $0.001$ and $0.999$. These parameters are chosen for visual purposes. The quantile function of $\mu_n$ has two concave regions corresponding to the two peaks of the density. The Wasserstein projection $\hat{\mu}_n$ is approximately the $\cW_2$-projection of $\mu^*$ onto $\cF_{m, 2}$ (see Theorem \ref{thm:consistency}). On the other hand, Grenander's estimator is approximately the projection with respect to the Kullback--Leibler divergence. In Figure \ref{fig:monotone3}, we show their quantile and density functions. Both estimators track closely the right tail of the true density. By construction, the Wasserstein projection provides a better overall fit of the empirical quantile function with respect to the $L^2$ distance, while Grenander's estimator returns its greatest convex minorant. Observe that the smaller density of $\hat{\mu}_n$ on the left end leads to a closer fit of the quantile function around the interval $[0.5, 0.8]$ (this effect is not apparent from the density plot). On the other hand, Grenander's estimator traces $Q_0$ more closely around $u \approx 0.45$ and $u \approx 0.9$ where $Q_0$ turns upward. This shows that the two estimators lead to different tradeoffs.
\end{example}

\subsection{Log-concave density estimation} \label{sec:implement.log.concave}
We assume that the data values $y_i$ are not all the same as otherwise the projection 
$\hat{\mu}_n$ is a point mass. We optimize over the convex subset $\tilde{\cQ}_{lc}$ of bounded quantile functions $Q$ in $\cQ_{lc}$ such that $h := 1/Q'$ is positive, concave and piecewise affine with respect to the partition $\Pi$ (see Lemma \ref{lem:log.concavity} and Theorem \ref{thm:log.concave.characterization}). The reciprocal $1/h$ makes log-concave density estimation more complicated than monotone density estimation. Nevertheless, the discretized optimization problem is still convex. 

Let $c := Q(0) \in \bR$ and $h_i := h(u_i)$, $i = 0, \ldots, K$. Then, ${\bf h} := (h_i)_{i = 0}^K$ satisfies the positivity constraint
\begin{equation} \label{eqn:h.constraint1}
h_i > 0, \quad i = 0, \ldots, K,
\end{equation}
and the concavity constraint
\begin{equation} \label{eqn:h.constraint2}
\frac{h_i - h_{i-1}}{u_i - u_{i-1}} \geq \frac{h_{i+1} - h_{i}}{u_{i+1} - u_i}, \quad i = 1, \ldots, K - 1.
\end{equation}
Conversely, given $c \in \bR$ and ${\bf h} := (h_i)_{i = 0}^K$ satisfying \eqref{eqn:h.constraint1}--\eqref{eqn:h.constraint2}, we define a concave function $h_{{\bf h}}: [0, 1] \rightarrow (0, \infty)$ by linear interpolation, and $Q_{c, {\bf h}} \in \tilde{Q}_{lc}$ by
\begin{equation*}
Q_{c, {\bf h}}(u) := c + \int_0^u
\frac{1}{h_{{\bf h}}(s)} \dd s, \quad u \in [0, 1],
\end{equation*}
and $\nu_{c, {\bf h}} := (Q_{c, {\bf h}})_{\#} \mathrm{Unif}(0, 1)$. Let $q_i := Q_{c, {\bf h}}(u_i)$, $i = 0, \ldots, K$. Note that $q_0 = c$. For $u_{i-1} \leq u \leq u_i$, we have
\begin{equation} \label{eqn:Q.c.h}
\begin{split}
Q_{c, {\bf h}}(u) 
&= q_{i-1} + \frac{u - u_{i-1}}{h_{i-1}} \Phi\left( \frac{u - u_{i-1}}{u_i - u_{i-1}} \delta_i \right),
\end{split}
\end{equation}
where $\delta_i := h_i/h_{i-1} - 1 > -1$  and $\Phi: (-1, \infty) \rightarrow (0, \infty)$ is the convex decreasing function defined by
\[
\Phi(\eta) := 
\left\{\begin{array}{ll}
(1/\eta) \log (1 + \eta), & \text{if } \eta \neq 0; \\
1, & \text{if } \eta = 0.
\end{array}\right.
\]
In particular, $Q_{c, {\bf h}}$ is affine on $[u_{i-1}, u_i]$ if $\delta_i = 0$ ($h_{i-1} = h_i$), and $q_i = q_{i-1} + (\Delta u_i / h_{i-1}) \Phi(\delta_i)$. Inverting \eqref{eqn:Q.c.h} and differentiating, we see that the density of $\nu_{c, {\bf h}}$ is given by
\begin{equation} \label{eqn:log.concave.estimated.density}
f_{{\bf c, {\bf h}}}(x) = h_{i-1} e^{\frac{h_{i-1} \delta_i}{\Delta u_i} (x - q_{i-1})}, \quad q_{i-1} < x < q_i.
\end{equation}

Next, we compute
\begin{equation*}
\begin{split}
\cW_2^2(\nu_{c, {\bf h}}, \mu_n) &
= \sum_{i = 1}^K \int_{u_{i-1}}^{u_i} (Q_{c, {\bf h}}(u) - y_i)^2 \dd u = \sum_{i = 1}^K (I_{i,1} + I_{i,2} + I_{i,3}),
\end{split}
\end{equation*}
where
\begin{equation} \label{eqn:log.concave.W2.integrals}
\begin{split}
I_{i,1} &:= (q_{i-1} - y_i)^2 \Delta u_i,\\
I_{i,2} &:= \frac{2(q_{i-1} - y_i) (\Delta u_i)^2}{h_{i-1}} \int_0^1 \lambda \Phi(\delta_i \lambda) \dd \lambda, \\
I_{i,3} &:= \frac{(\Delta u_i)^3}{h_{i-1}^2}  \int_0^1 \lambda^2 (\Phi(\delta_i \lambda))^2 \dd \lambda
\end{split}
\end{equation}
are smooth and convex functions of $(c, {\bf h})$.\footnote{Explicit expressions of the integrals in $I_{i,2}$ and $I_{i,3}$ are omitted.} We implement the convex program
\begin{equation} \label{eqn:log.concave.convex.program}
\min_{c, {\bf h}} \cW_2^2(\nu_{c, {\bf h}}, \mu_n) \quad  \text{subject to \eqref{eqn:h.constraint1} and \eqref{eqn:h.constraint2}}
\end{equation}
simply by the general purpose optimization R package \texttt{nloptr} \cite{nloptr}. We implement the univariate log-concave MLE using the R package \texttt{logcondens} \cite{longcondens}. It is well known that the density of the MLE is supported on the convex hull of the data. Moreover, it is piecewise log-affine and each knot occurs at one of the data points.

\begin{figure}[t!]
\includegraphics[scale = 0.48]{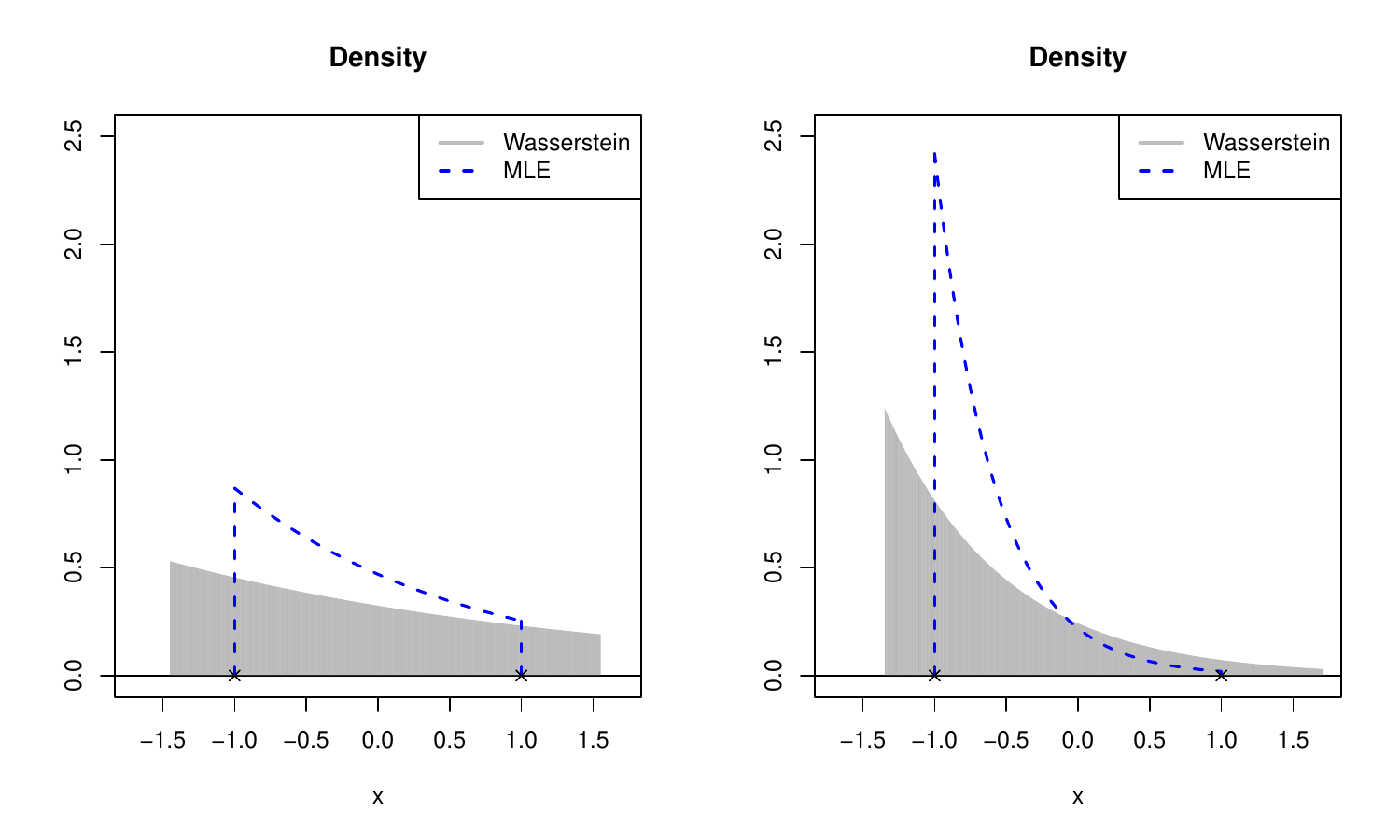}
\vspace{-0.4cm}
\caption{Estimated densities for the two-point distribution in Example \ref{eg:logconcave1}.
Left: $\lambda = 0.4$. Right: $\lambda = 0.2$.} \label{fig:logconcave1}
\end{figure}

\begin{example}[Mixture of two points] \label{eg:logconcave1}
For $\lambda \in [0, 1]$, let $\mu_n$ be the mixture distribution
\[
\mu_n = (1 - \lambda) \delta_{-1} + \lambda \delta_1.
\]
When $\lambda \in \{0, 1\}$, both the Wasserstein projection estimator and the MLE reduce to a point mass. When $\lambda = 0.5$, the maximum likelihood returns $\mathrm{Unif}(-1, 1)$, and the Wasserstein projection estimator gives $\mathrm{Unif}(-1.5, 1.5)$ as seen in Example \ref{eg:log.concave.uniform}. In Figure \ref{fig:logconcave1} we plot the estimated density when $\lambda \in \{0.4, 0.2\}$. While the MLE always has support $[0, 1]$ (when $\lambda \neq 0, 1$), we observe that the density of the Wasserstein projection estimator has a wider support in both cases.
\end{example}

\begin{figure}[t!]
\includegraphics[scale = 0.48]{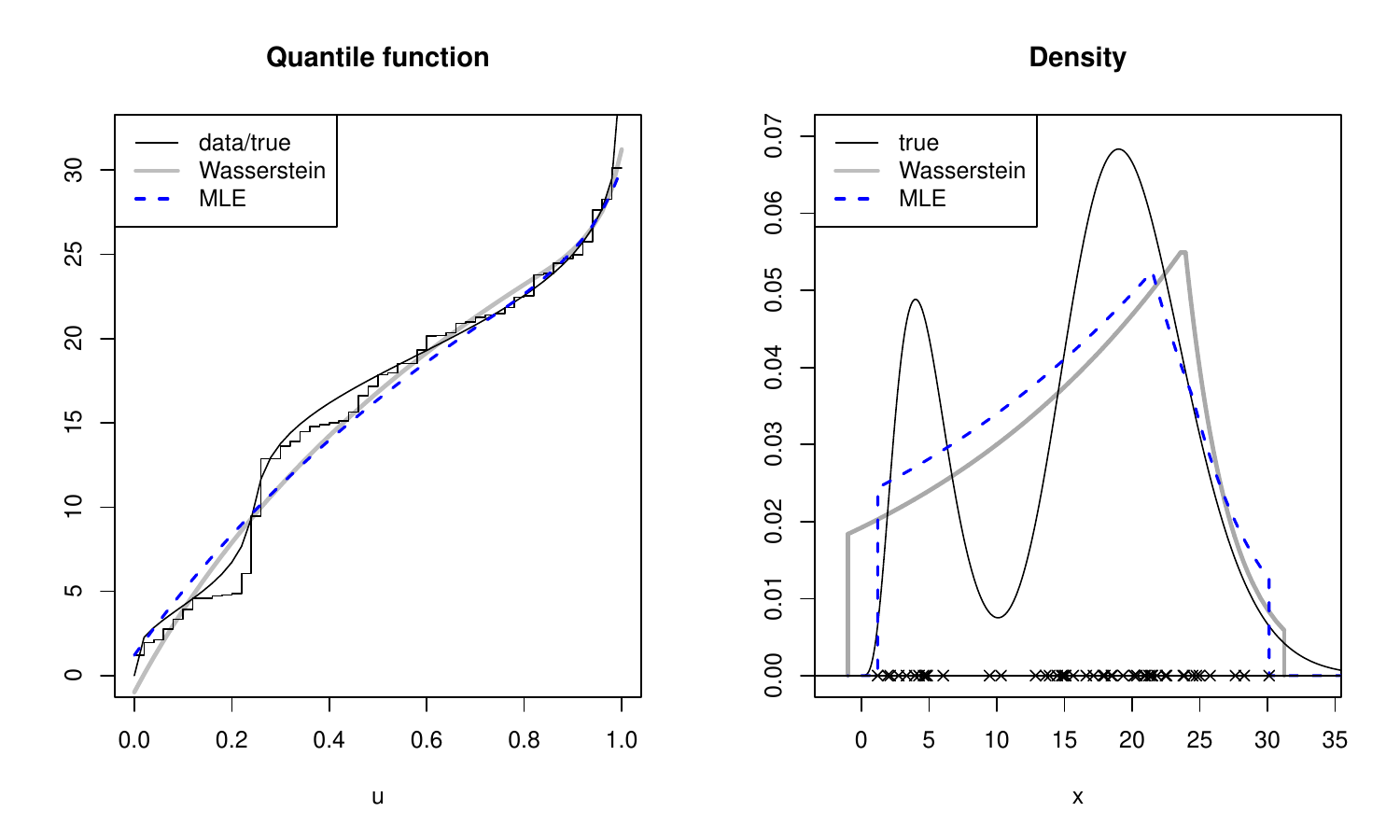}
\vspace{-0.4cm}
\caption{Left: Empirical, true and estimated quantile functions in the context of Example \ref{eg:logconcave2}. Right: True and estimated densities. The data points are shown by the crosses.} \label{fig:logconcave2}
\end{figure}

\begin{example}[A misspecified case] \label{eg:logconcave2}
Let $\mu_n$ be the empirical distribution of $n = 50$ i.i.d.~samples from $\mu^* = 0.25 \mathrm{Gamma}(5, 1) + 0.75 \mathrm{Gamma}(20, 1)$ (again the parameters are chosen for illustration purposes). Being bimodal, $\mu^*$ is not log-concave. The estimated quantile functions and the densities are shown in Figure \ref{fig:logconcave2}. The maximum likelihood estimate and the Wasserstein projection estimate are qualitatively quite similar, except that the support of the former is the convex hull of the data, while that of the latter is slightly wider. The fact that the true distribution is supported on $\bR_+$ is not used in the Wasserstein projection.\footnote{If desired, projection to $\cF_{lc} \cap \cP_2(\bR_+)$ can be easily enforced in the optimization problem \eqref{eqn:log.concave.convex.program} by adding the constraint $c = Q_{c, {\bf h}}(0) = \min \mathrm{support}(\nu_{c, {\bf h}}) \geq 0$.} Without imposing this constraint, the estimated density gives a small amount of mass on the left hand side of zero.
\end{example}

\section{Discussion}  \label{sec:discussion}
In this paper, we study shape-constrained density estimation by Wasserstein projection in the univariate setting. We focus on monotone and log-concave density estimation which are some of the most important shape constraints used in practice. In both cases, we establish fundamental results in the univariate setting and provide empirical examples. Nevertheless, many more questions arise:
\begin{enumerate}
    \item[1.] In Theorems \ref{thm:monotone.characterization} and Theorem \ref{thm:log.concave.characterization}, we showed that the estimated density is piecewise affine/log-concave. Our proof techniques do not provide further information about the number of break points and their locations. Improved understanding of these break points will likely suggest more precise and efficient algorithms. 
    \item[2.] Further statistical properties of the Wasserstein projection estimator. Here, finite sample properties are only investigated in Wasserstein distance when $p = 2$. A natural direction is to investigate convergence properties of the estimated density. 
    \item[3.] Interpolation between Wasserstein projection and Kullback--Leibler (maximum likelihood) projection. For example, one may consider an interpolating distance between Wasserstein (Otto) and {F}isher--{R}ao metrics \cite{CPSV18}. It is also natural to ask if projecting with respect to the entropically regularized Wasserstein distance, also called the Sinkhorn divergence, leads to different or improved behaviors. In particular, it was observed in \cite{RW18} that entropic optimal transport is related to maximum likelihood deconvolution. A Riemannian geometry induced by the Sinkhorn divergence was introduced recently in \cite{LLMST24}.
    \item[4.] Shape-constrained estimation of multivariate distributions by Wasserstein projection, especially in the log-concave setting. Since the space of log-concave distributions on $\bR^{d}$ is not displacement convex for $d \geq 2$ \cite{santambrogio2016convexity}, it is not immediate whether the Wasserstein projection is well-posed. Even if this is the case, computation of the Wasserstein projection is non-trivial due to the curvatures of the Wasserstein space $(\cP_2(\bR^d), \cW_2)$ and the model. For further discussions of Wasserstein projection see \cite{PSW25, VCBK25} and the references therein.
\end{enumerate}

\section*{Acknowledgment}
Takeru Matsuda was supported by JSPS KAKENHI Grant Numbers 21H05205, 22K17865 and 24K02951 and JST Moonshot Grant Number JPMJMS2024. Ting-Kam Leonard Wong acknowledges support by the NSERC Discovery Grant RGPIN-2025-06021. He thanks Piotr Zwiernik, Ricardo Baptista and Bodhisattva  Sen for helpful discussions.

\appendix

\bibliographystyle{plain}
\bibliography{references} 
\end{document}